\documentclass[10pt,reqno]{amsart}

\usepackage[latin1]{inputenc}
\usepackage[T1]{fontenc}
\usepackage{amsmath,amsfonts,amssymb,amsthm,cases,mathtools,mathrsfs}
\usepackage{enumerate}
\usepackage{geometry}
\usepackage{array,dcolumn,booktabs}
\usepackage[colorlinks=true,linkcolor=blue,citecolor=red]{hyperref} 

\usepackage{tikz}
\usetikzlibrary{patterns,positioning,arrows,shapes}

\usepackage{etoolbox}

\makeatletter
\patchcmd{\@maketitle}
  {\ifx\@empty\@dedicatory}
  {\ifx\@empty\@date \else {\vskip3ex \centering\footnotesize\@date\par\vskip1ex}\fi
   \ifx\@empty\@dedicatory}
  {}{}
\patchcmd{\@adminfootnotes}
  {\ifx\@empty\@date\else \@footnotetext{\@setdate}\fi}
  {}{}{}
\makeatother

\makeatletter
\newsavebox\myboxA
\newsavebox\myboxB
\newlength\mylenA

\newcommand*\xoverline[2][0.75]{%
    \sbox{\myboxA}{$\m@th#2$}%
    \setbox\myboxB\null
    \ht\myboxB=\ht\myboxA%
    \dp\myboxB=\dp\myboxA%
    \wd\myboxB=#1\wd\myboxA
    \sbox\myboxB{$\m@th\overline{\copy\myboxB}$}
    \setlength\mylenA{\the\wd\myboxA}
    \addtolength\mylenA{-\the\wd\myboxB}%
    \ifdim\wd\myboxB<\wd\myboxA%
       \rlap{\hskip 0.5\mylenA\usebox\myboxB}{\usebox\myboxA}%
    \else
        \hskip -0.5\mylenA\rlap{\usebox\myboxA}{\hskip 0.5\mylenA\usebox\myboxB}%
    \fi}
\makeatother

 \author{Victor Vila\c{c}a Da Rocha}
\address{Laboratoire de Math\'ematiques Jean Leray, Universit\'e de Nantes, UMR CNRS 6629, \newline
2, rue de la Houssini\`ere, 44322 Nantes Cedex 03, France}
\email{victor.vilaca-da-rocha@univ-nantes.fr}
\date{September 14, 2016}

\title[Modified scattering and beating effect for Schr\"{o}dinger systems on product spaces]
{Modified scattering and beating effect for coupled Schr\"{o}dinger systems on product spaces with small initial data}

\newtheorem{thm}{Theorem}[section]
\newtheorem{lem}[thm]{Lemma}

\newtheorem{rem}[thm]{Remark}
\newtheorem{cor}[thm]{Corollary}

\DeclareMathOperator{\sgn}{sgn}

\numberwithin{equation}{section}  

\begin{document}

\begin{abstract}
In this paper, we study a coupled nonlinear Schr\"{o}dinger system with small initial data in a product space.
We establish a modified scattering of the solutions of this system and we construct a modified wave operator.
The study of the resonant system, which provides the asymptotic dynamics, allows us to highlight a control of
the Sobolev norms and interesting dynamics with the beating effect.
The proof uses a recent work of Hani, Pausader, Tzvetkov and Visciglia for the modified scattering, and a recent
work of Gr\'ebert, Paturel and Thomann for the study of the resonant system.
\end{abstract}

\subjclass[2010]{35Q55}
\keywords{Modified scattering, Modified wave operator, Beating effect, Nonlinear Schr\"{o}dinger equation.}
\thanks{Work partially supported by the grant ANA\'E, ANR-13-BS01-0010-03.}

\maketitle

\section{Introduction}

The purpose of this paper is to study the asymptotic behavior of the following cubic defocusing coupled nonlinear Schr\"odinger system with small initial data:
\begin{equation}\begin{dcases}
i\partial_tU+\Delta_{\mathbb{R}\times\mathbb{T}}U&=\left|V\right|^2U, \quad (t,x,y)\in [0,+\infty)\times\mathbb{R}\times\mathbb{T} \\ 
i\partial_tV+\Delta_{\mathbb{R}\times\mathbb{T}}V&=\left|U\right|^2V,
\end{dcases}\label{sys}\end{equation}
where $\mathbb{T}=\mathbb{R}/2\pi\mathbb{Z}$ is the one dimensional torus and $U,V$ are complex valued functions on the spatial product space 
$\mathbb{R}\times\mathbb{T}$.
We construct some solutions of this system which exhibit a modified scattering to solutions of a nonlinear resonant system.
Thanks to the nonlinearity, we construct solutions which exhibit a beating effect, namely a transfer of energy between two different modes of the couple of 
solutions. This is a genuine nonlinear behavior.

\subsection{Motivations and background}

When dealing with small solutions of nonlinear evolution equations, there are usually two main axes of research. 
The first one is the linear approach. The idea is to show that for small initial data, the solutions of these equations tend to be close to solutions of the
associated linear equations.
The second way to analyze these equations is the nonlinear approach. This is the approach we choose in this paper. 
This time, the goal is to find some solutions with a nonlinear behavior, i.e a behavior that doesn't exist for the linear equation. 
From this point of view, we can expect a large range of results, depending on the kind of nonlinear behavior we want to highlight. 
First, we can expect different kinds of nonlinear behavior.
For example, we can expect some growth of the Sobolev norms, or even a blow-up of the solutions. 
We can also expect some interaction between the modes and frequencies of the solutions (e.g. some exchange of energy), existence of solitons... 
Then, nonlinear effects can appear on different time scales.
This can be finite time behavior, large but finite time behavior depending of the size of the initial data, or even infinite time behavior.
Finally, we can make a new dichotomy with the way the solutions are enjoying the behavior: 
are the solutions presenting themselves the nonlinear behavior, or are they scattering to solutions of other equations which present this behavior ?

\bigskip

From the point of view of the search of instability, the cubic nonlinear Schr\"odinger equation is a perfect candidate due to this cubic nonlinearity that 
offers some resonant opportunities. Physically, the cubic nonlinear Schr\"odinger equation
\begin{equation*}
 i\partial_tU+\frac12\Delta U= \lambda |U|^2U,
\end{equation*}
with $\lambda\in\mathbb{R}$, appears in a large range of phenomena, like for example, the propagation of light in nonlinear optical fibers,
the Bose-Einstein condensates theory, the gravity waves, the water waves, the plasma oscillations in magnetohydrodynamics\ldots
This duality between the physical and mathematical interest makes this equation one of the most studied and one of the most important models in nonlinear science.
The kind of nonlinear behavior we can obtain depends on the geometry of the spatial domain. Let us present some of these nonlinear results.

For Euclidean spaces, we can show that the solutions exist globally, are decreasing and exhibit some kind of scattering to free solutions.
We refer to Kato and Pusateri in~\cite{KP} for the critical case of the space $\mathbb{R}$ (in this case the scattering result is modified), 
and to Hayashi and Naumkin in \cite{HN} for the more general case of the Euclidean space~$\mathbb{R}^n$. Both results are extensions of initial works of 
Ozawa for the one dimensional case (\cite{Oza91}), and Ginibre and Ozawa for the higher dimensions (\cite{GinOza93}).

For hyperbolic spaces, they are also a lot of scattering results for the cubic Schr\"odinger equation.
For example, Banica, Carles and Staffilani show in \cite{BCS} that we have a scattering behavior and a wave operator in~$H^1(\mathbb{H}^3)$.
Higher dimension cases are treated, for smaller nonlinearities and always in $H^1$, by Banica, Carles and Duyckaerts with radial assumptions (\cite{BanCarDuy09}), 
or by Ionescu, Pausader and Staffilani without radial assumptions~(\cite{IonPauSta12}).

For compact domains, in the case of the torus, we can use the resonances between the modes of the solutions to exhibit some growth of the Sobolev norms
(see Colliander, Keel, Staffilani, Takaoka and Tao in the case of the torus $\mathbb{T}^2$ in \cite{CKSTT}). We can also use these resonances to see that
the equation is strongly illposed on the Sobolev spaces $H^s$, with $s<0$, in the torus $\mathbb{T}^d$ for $d\geq2$ (see Carles and Kappeler in \cite{CK});
or to see the equation is instable with respect to the initial value (see Carles, Dumas and Sparber in \cite{CDS}).

Finally, for product spaces, we can expect to mix some of these geometric domains to obtain two different kinds of nonlinearity at the same time.
The initial idea of Tzvetkov and Visciglia in \cite{TV} is to study the equation on the space $\mathbb{R}^n\times\mathcal{M}$, with $\mathcal{M}$
a compact Riemannian manifold. They show that the solutions exist globally and scatter to free solutions for small initial data. 
The two authors extend their result to the large data case with the compact manifold $\mathcal{M}=\mathbb{T}$ in \cite{TV2}. They obtain again a 
global existence and a scattering behavior. As in the Euclidean spaces case, the problem is here that solutions of the cubic Schr\"odinger equation
scatter to free solutions, which prevents from all the compact-kind nonlinear behavior such as a growth of the Sobolev norms.

\bigskip

From the point of view of the product space, Hani, Pausader, Tzvetkov and Visciglia in \cite{HPTV} avoid this problem by considering the space 
$\mathbb{R}\times\mathbb{T}^d$, for $1\leq d\leq4$. The idea is to keep a direction of diffusion with $\mathbb{R}$ to enable scattering results
and to add the compact manifold $\mathbb{T}^d$ to obtain interesting nonlinear behaviors.
The result here is very interesting because they show that, for small initial data, 
the solutions of the cubic Schr\"odinger equation scatter to solutions of another equation called the resonant equation, instead of free solutions.
They also obtain a modified wave operator from this resonant equation to the cubic Schr\"odinger equation.
As a consequence, we can expect a whole new range of nonlinearity mixing the scattering theory from the Euclidean part, and all the kind of nonlinearity we 
can highlight for the resonant equation. In particular, they show how to transfer solutions of a reduced resonant equation on the torus $\mathbb{T}^d$ to solutions
of the resonant equation. Therefore, we can expect all 'torus-kind' of nonlinearity (thanks to the resonances) for the resonant equation, as a growth of the Sobolev
norms, and thus find solutions to the the cubic Schr\"odinger equation that scatter to these solutions. 

An interesting feature of this method employed by Hani, Pausader, Tzvetkov and Visciglia, is the fact that it is adaptable. 
From example, adding a convolution potential in order to kill the resonances, Gr\'ebert, Paturel and Thomann show in \cite{GPT2} that the modified equation 
admits modified scattering and a wave operator too, but they show that all the Sobolev norms of their solutions tend to be constant.
Another example is given by Hani and Thomann in \cite{HT}. They add a harmonic trapping and obtain once again a modified scattering result and the existence of the 
modified wave operator. The important fact here is that the resonant dynamics allow them to justify and extend some physical approximations in the theory of 
Bose-Einstein condensates in cigar-shaped traps. Finally, a last example is given by Haiyan Xu in \cite{Xu}.
By modifying the equation, she establishes a scattering theory between the cubic Schr\"odinger equation on the cylinder $\mathbb{R}\times\mathbb{T}$
and the Szeg\H{o} equation. Thanks to the study of the Szeg\H{o} equation (see for example the article \cite{GerGre16} of G\'erard and Grellier), 
this allows the author to construct global unbounded solutions to this modified cubic Schr\"odinger equation.

As in the three previous examples, the goal of this paper is to transport the method of Hani, Pausader, Tzvetkov and Visciglia to another problem, the study of the
cubic coupled Schr\"odinger system. Indeed, the cubic coupled Schr\"odinger systems present some interesting dynamical properties we can present now.

\bigskip

From the mathematical point of view, the study of systems offers more possibilities than a single equation. However, the first step is often to check that we can
extend the results of the equation case study to the system case. In that optic, we can see from a general point of view on evolution systems that the presence of
a direction of diffusion seems to allow scattering. For example, for the Klein Gordon equation on product spaces, 
Hari and Visciglia in \cite{HariVisc} obtain scattering of the solutions to free solutions for small initial data.
In the case of Schr\"odinger systems, Cassano and Tarulli show in \cite{CassTaru} a scattering result and the existence of the wave operator, but this study
doesn't allow to look at the cubic case when there is just one Euclidean space direction. Let us now deal with cubic coupled Schr\"odinger systems.

As the cubic Schr\"odinger equation, the cubic coupled Schr\"odinger system presents at the same time a physical and a mathematical interest.
Physically, the cubic coupled Schr\"odinger system occurs in nonlinear optics while, for example, 
looking at two orthogonally polarized components traveling at different speeds because of different refractive indices associated with them.
We can also study the coupling between two different optical waveguides, that can be provided by a dual-core single-mode fiber. 
Another example is given by two distinct pulses with different carrier frequencies but with the same polarization. 
For more precisions on these examples, we can refer to the book \cite{AB} of G. Agrawal and R. Boyd on nonlinear optics.

Of course, the kind of result we obtain for systems depends on the geometry of the space domain. 
For Euclidean spaces, the role of the coupling is not predominant. Basically, we obtain same kind of results as in the one equation case.
As in the Hayashi-Naumkin article \cite{HN}, the Euclidean direction can provide some decrease of the $L^\infty$ norms. For example,
Donghyun Kim studies in \cite{Kim} a cubic coupled Schr\"odinger system with different mass. He obtains the global existence and a decay of the solutions.
Another example is given by Li and Sunagawa in \cite{LiSun16}, where the authors study the decay rate of the $L^\infty$ norms for a large range of cubic 
nonlinearities, including the derivative nonlinearities.
The diffusion direction allows extension of the results of Kato and Pusateri (\cite{KP}) in the case of a cubic coupled Schr\"odinger system on $\mathbb{R}$,
with global existence, a scattering result and a decreasing of the solutions (see \cite{VVDR1}). For bigger Euclidean spaces, we can cite the recent work of 
Farah and Pastor who show global existence and scattering in $H^1(\mathbb{R}^3)$ in \cite{FaraPast}.

For the torus case, the study of systems allows to create more nonlinear behaviors by creating mixing between the different modes of the solutions.
This is the main interest of the study of the cubic coupled Schr\"odinger systems. For example, Gr\'ebert, Paturel and Thomann in \cite{GPT} highlight a beating
effect for cubic coupled Schr\"odinger systems on $\mathbb{T}$. This beating effect consists in an exchange of energy between different modes of the solution.
This surprising behavior can be observed by simple experiences, for example, with two identical clothespins on a wire. In this experience, with a small perturbation,
we can observe a beating effect between the two clothespins (see the video \cite{G_PAL}, in French).
The existence of the beating effect for the cubic coupled Schr\"odinger systems, based on a Birkhoff normal form decomposition of the Hamiltonian of the system, 
is thus proved for large but finite time.
This kind of nonlinear behavior is made possible by the mixing effect mentioned above.

\bigskip

The goal of this article is to obtain a beating effect behavior for the cubic coupled Schr\"odinger systems in infinite time.
For this purpose, two main solutions may be envisioned: the KAM theory or the use of scattering results.

Concerning the KAM theory, an example is given by the study of the quintic Schr\"odinger equation on the torus.
First, Gr\'ebert and Thomann obtained for this equation a beating effect for finite time in~\cite{GreTho12}.
Thanks to a KAM theorem, Haus and Procesi then constructed quasi-periodic solutions that exchange energy, through the beating effect, 
over an infinite time (\cite{HauPro16}). 
The advantage of this method is the fact that it allows to stay on the torus in order to obtain the infinite time behavior, 
when we need for the scattering method to add a dispersive direction.

In this paper, we choose the second method and thus consider the spatial domain $\mathbb{R}\times\mathbb{T}$. 
The goal is to use the diffusion direction to get a modified scattering result and construct a modified wave operator
by using the method of Hani, Pausader, Tzvetkov and Visciglia; and to use the compact component $\mathbb{T}$ to obtain interesting dynamics such as the beating
effect of Gr\'ebert, Paturel and Thomann. By mixing these two approaches, we obtain a couple of solutions of our system~\eqref{sys} that scatters to a solution
of a resonant system that provides a beating effect.
Moreover, we show that the construction of the modified scattering implies some dynamical consequences, such as the control of the Sobolev norms.

\subsection{The results}

In this section, we present the main theorems of the paper. 
In the following results, we denote by $\mathcal{F}H(\xi,p)=\hat{H}(\xi,p)=\hat{H}_p(\xi)$ the Fourier transform of 
$H:\mathbb{R}\times\mathbb{T}\longrightarrow\mathbb{C}$ at the point $(\xi,p)\in\mathbb{R}\times\mathbb{Z}$.
Moreover, the norm $S^+$ we consider here is a strong $L^2$ based norm introduced in Section \ref{secprel} and $N\geq12$ is an integer. 
We recall that we study the system
\begin{equation*}\begin{dcases}
i\partial_tU+\Delta_{\mathbb{R}\times\mathbb{T}}U&=\left|V\right|^2U, \quad (t,x,y)\in [0,+\infty)\times\mathbb{R}\times\mathbb{T} \\ 
i\partial_tV+\Delta_{\mathbb{R}\times\mathbb{T}}V&=\left|U\right|^2V.
\end{dcases}\end{equation*}
Due to the approach of Hani, Pausader, Tzvetkov and Visciglia in \cite{HPTV}, we want to make a
link between the behavior of the system \eqref{sys} and a resonant system we define here:
\begin{equation}\begin{dcases}
i\partial_{\tau}W_U(\tau)=\mathcal{R}\left[W_V(\tau),W_V(\tau),W_U(\tau)\right],\\
i\partial_{\tau}W_V(\tau)=\mathcal{R}\left[W_U(\tau),W_U(\tau),W_V(\tau)\right],
\end{dcases}\label{res}\end{equation}
with 
\begin{equation*}
\mathcal{F}\mathcal{R}\left[F,G,H\right](\xi,p)=\displaystyle\sum\limits_{\substack{p,q,r,s\in\mathbb{Z} \\ 
p-q+r-s=0 \\ p^2-q^2+r^2-s^2=0}}\hat{F}(\xi,q)\xoverline{\hat{G}(\xi,r)}\hat{H}(\xi,s). 
\end{equation*}
The links between these two systems are presented in Section \ref{modanswav} through the modified scattering and modified wave operator theorems. Thanks to these
theorems, we obtain the two main results of this paper.

First, by using the modified scattering operator of Theorem \ref{modscat}, we have a control of all the Sobolev norms with the following theorem:

\begin{thm}\label{thmbounded}
There exists $\varepsilon>0$ such that if $U_0,V_0\in S^+$ satisfies 
\begin{equation*}
\|U_0\|_{S^+}+\|V_0\|_{S^+}\leq\varepsilon,
\end{equation*}
and if $(U(t),V(t))$ solves \eqref{sys} with initial data $(U_0,V_0)$, then $(U,V)\in\mathcal{C}([0,+\infty):H^N)\times\mathcal{C}([0,+\infty) : H^N)$ exists 
globally and, for all $s\in\mathbb{R}$, there exists a constant $M=M(s)>0$ such that we have
\begin{equation}\label{staysmall}
\|U(t)\|_{H^s_{x,y}}+\|V(t)\|_{H^s_{x,y}}\leq M\varepsilon.
\end{equation}
More precisely, there exists a constant $c=c(s)\geq0$ such that
\begin{equation}\label{eqboundhs}
 \lim_{t\rightarrow+\infty}\left(\|U(t)\|_{H^s_{x,y}}+\|V(t)\|_{H^s_{x,y}}\right)=c.
\end{equation}
\end{thm}

This theorem shows that for small initial data, the solutions stay small in every Sobolev spaces. 
Moreover, we see in the second part of the theorem that the sum of the Sobolev norms of the couples of solutions tends to be constant.

\bigskip

Then, thanks to the construction of the modified wave operator in Theorem \ref{waveop}, the idea is to find some interesting nonlinear behavior of the resonant
system \eqref{res} in order to transfer this behavior to the initial system. As a consequence, we follow the strategy of \cite{GPT} to construct couples of 
solutions of the initial system \eqref{sys} which scatter to beating effect solutions of the resonant system \eqref{res}:

\begin{thm}\label{beating}
Let $I\subset\mathbb{R}$ be a bounded open interval, $(p,q)$ a couple of different integers and $0<\gamma<\frac12$.
For $\varepsilon=\varepsilon(p,q,I)>0$ small enough, there exists:
\begin{itemize}
 \item a constant $0<T_\gamma\lesssim |\ln(\gamma)|$ and a $2T_\gamma-$periodic function $K_\gamma:\mathbb{R}\rightarrow(0,1)$ such that
 \begin{equation*}
  K_\gamma(0)=\gamma,\qquad \qquad \qquad K_\gamma(T_\gamma)=1-\gamma;
 \end{equation*}
 \item a couple of solutions $(W_U,W_V)$ of the resonant system \eqref{res} which exhibits a beating effect in the following sense:
 \begin{equation*}\begin{dcases} 
\hat{W}_U(t,\xi,y)=\mathcal{F}(W_U)(t,\xi,p)e^{ipy}+\mathcal{F}(W_U)(t,\xi,q)e^{iqy}, \\
\hat{W}_V(t,\xi,y)=\mathcal{F}(W_V)(t,\xi,p)e^{ipy}+\mathcal{F}(W_V)(t,\xi,q)e^{iqy}, 
\end{dcases}\end{equation*}
and $\forall\xi\in I$,
\begin{equation*}\begin{dcases}                               
|\hat{W}_V(t,\xi,p)|^2=|\hat{W}_U(t,\xi,q)|^2=\varepsilon^2K_\gamma(\varepsilon^2t),\\
|\hat{W}_U(t,\xi,p)|^2=|\hat{W}_V(t,\xi,q)|^2=\varepsilon^2(1-K_\gamma(\varepsilon^2t));
\end{dcases}
 \end{equation*}
 \item a couple of solutions $(U,V)$ of the initial system \eqref{sys} which exhibits modified scattering to this couple $(W_U,W_V)$ in the following sense:
 \begin{equation*}\begin{dcases}                               
\|U(t)-e^{it\Delta_{\mathbb{R}\times\mathbb{T}}}W_U(\pi\ln(t))\|_{H^N(\mathbb{R}\times\mathbb{T})}&\rightarrow0 \quad \text{as} \quad t\rightarrow+\infty,\\
\|V(t)-e^{it\Delta_{\mathbb{R}\times\mathbb{T}}}W_V(\pi\ln(t))\|_{H^N(\mathbb{R}\times\mathbb{T})}&\rightarrow0 \quad \text{as} \quad t\rightarrow+\infty.
\end{dcases}
\end{equation*}
\end{itemize}
\end{thm}

With this theorem, we see the importance of both parts of the product space. On the one hand, the bounded component allows us to construct couple of
solutions with very nonlinear behavior such as this beating effect (see \cite{GPT}). 
On the other hand, the Euclidean component gives an infinite time behavior (see for example \cite{VVDR1}).
Therefore, thanks to the product space, we have here an asymptotic convergence to the beating effect, which is a new result. The counterpart of this construction
is the fact that this behavior is local in the Euclidean coordinate.

\subsection{Overview of proofs}

\subsubsection{The modified scattering and the wave operator}
According to previous results in scattering theory for Schr\"odinger equations and systems 
(\cite{BCS},\cite{CassTaru},\cite{GPT2},\cite{HPTV},\cite{HT},\cite{HariVisc},\cite{HN},\cite{KP},\cite{TV},\cite{TV2},\cite{VVDR1},\cite{Xu}),
it is relevant to introduce the profiles $(F,G)$ of the solutions $(U,V)$, which are the backwards linear evolutions of solutions to the nonlinear equations:
\begin{equation}\label{defprof}
F(t,x,y):=e^{-it\Delta_{\mathbb{R}\times\mathbb{T}}}U(t,x,y) \qquad\qquad G(t,x,y):=e^{-it\Delta_{\mathbb{R}\times\mathbb{T}}}V(t,x,y).  
\end{equation}
The system described by the profiles looks like 
\begin{equation*}\begin{dcases}
i\partial_tF(t)&=\mathcal{N}^t[G(t),G(t),F(t)],\\ 
i\partial_tG(t)&=\mathcal{N}^t[F(t),F(t),G(t)].
\end{dcases}\end{equation*}

To isolate a resonant system, the idea is to work on the structure of the nonlinearity $\mathcal{N}^t$. According to a stationary phase intuition,
we decompose the nonlinearity as 
\begin{equation*}
 \mathcal{N}^t=\frac{\pi}{t}\mathcal{R}+\mathcal{E}.
\end{equation*}
The integrable part $\mathcal{E}$ enjoys a fast decrease. The idea is to show that it doesn't play a role in the asymptotic dynamics of $F$ and $G$.
Thus, the system described by $\mathcal{R}$ is our resonant system, the system which contains all the asymptotically dynamics. The goal is thus to find interesting
dynamics such as the beating effect for this resonant system. For that purpose, we construct a new system, the reduced resonant system, which is obtained from
the resonant system by deleting the Euclidean variable. This reduced resonant system lives on the torus $\mathbb{T}$, where we have the beating effect thanks to 
\cite{GPT}. Therefore, we have to show how to transfer solutions of the reduced resonant system to solutions of the resonant system.

Thanks to the fast decrease of the $\mathcal{E}$ part of the nonlinearity, fixed-point arguments allow us to prove the existence of the modified scattering
and modified wave operators.

\subsubsection{The dynamical consequences}
Both theorems (the modified scattering and the modified wave operator) imply dynamical consequences for the initial system \eqref{sys}.

The modified scattering theorem shows that all solutions of the initial system \eqref{sys} scatter to solutions of the resonant system \eqref{res}.
Thus, proving that all the solutions of the resonant system \eqref{res} are bounded in Sobolev spaces, we obtain that all the solutions of the 
initial system \eqref{sys} are also bounded in Sobolev spaces. This prevents any kind of growth of Sobolev norms for our cubic coupled Schr\"odinger system
on the product space $\mathbb{R}\times\mathbb{T}$.

The modified wave operator theorem shows that for all solution of the resonant system \eqref{res}, there exists a solution of the initial system \eqref{sys}
which scatters to this resonant system solution. Thanks to this operator, we can transfer the beating effect from the resonant system to the initial system
and thus obtain solutions of the system \eqref{sys} which scatter to beating effect solutions of the resonant system.

This strategy is summarized in the following schema:

\begin{center}
      \begin{tikzpicture}[node distance=3cm and 4cm]
       
       \draw node[ultra thick, align=center, rectangle, draw=blue, scale=1.5, text width=1.8cm] (Sys) {Coupled Schr\"odinger system};

       \draw node[ultra thick, align=center, rectangle, draw=blue, scale=1.5, above=of Sys, text width=1.8cm] (Res) {Resonant system};
   
       \draw [-triangle 90] (Sys) to [bend right] node [align=center, right,text width=1.8cm] {Modified scattering} (Res);
      
       \draw [-triangle 90] (Res) to [bend right] node [align=center, left, text width=1.8cm] {Wave operator} (Sys);
      
       \draw node [ellipse, scale=1.2, draw=red, align=center, right=of Res, text width=1.8cm, xshift=-2cm] (BSres) {Bounded solutions} ;
 
       \draw node [ellipse, scale=1.2, draw=green, align=center, right=of Sys, text width=1.8cm, xshift=-2cm] (BSsys) {Bounded solutions};
      
       \draw node[ultra thick, align=center, rectangle, draw=blue, scale=1.5, left=of Res, text width=1.8cm] (Red) {Reduced resonant system};

       \draw [-triangle 90] (Red) --  node [align=center,above, text width=1.8cm] {Transfer of}
       node [align=center,below, text width=1.8cm] {solutions}(Res);
          
       \draw node [ellipse, scale=1.2, draw=red, align=center, below=of Red, text width=1.8cm,yshift=2.3cm] (Beatred) {Beating effect};
   
       \draw node [ellipse, scale=1.2, draw=green, align=center, left=of Sys, text width=1.8cm, xshift=3cm] (Beatsys) {Beating effect};
       
       \draw [dashed,-triangle 90] (Beatred) --  node [align=center,left, text width=3.5cm] {\newline Transfer of solution and Wave Operator} (Beatsys);
       
       \draw [dashed,-triangle 90] (BSres) --  node [align=center,right, text width=1.8cm] {Modified scattering} (BSsys);
       
       \draw [] (BSres) -- (Res);
       
       \draw [] (BSsys) -- (Sys);
       
       \draw [] (Beatred) -- (Red);
       
       \draw [] (Beatsys) -- (Sys);

      \end{tikzpicture}
\end{center}

\subsection{Plan of the paper}
In Section \ref{secprel}, we introduce the different norms and notations we need and state some preliminary estimates. 
In Section \ref{modanswav}, we present the modified scattering and wave operator results, which are extensions of the results of \cite{HPTV}.
Finally, the resonant system is introduced and studied in Section~\ref{Secres}. In particular, we obtain here the boundedness of the solutions and we
construct the solutions which provide some beating effect.

\section{Preliminaries}\label{secprel}
In this section, we first introduce in Subsection \ref{subnorm} all the notations and norms we use in the paper. Then, in Subsection \ref{subprofi},
we introduce the profiles and the nonlinearity associated to the profile system.

\subsection{Norms and notations}\label{subnorm}

\subsubsection{Some notations}
We mainly follow the notations of \cite{HPTV}.

Concerning standard notations, we use the notation $f\lesssim g$ to denote that there exists a constant $c>0$ such that $f\leq cg$. 
This notation allows us to avoid dealing with all the constants in the different inequalities. We also use the usual notation
$\left<p\right>:=\sqrt{1+p^2}$.

Most of the time, we use the distinction between lower case letters and capitalized letters to specify on which spatial domain the functions are defined.
On the one hand, we use lower case letters to denote functions of the Euclidean variable $f:\mathbb{R}\rightarrow\mathbb{C}$
or sequences $a:\mathbb{Z}\rightarrow\mathbb{C}$. On the other hand, we use capitalized letters to denote functions on the product space
$F:\mathbb{R}\times\mathbb{T}\rightarrow\mathbb{C}$.

We define the spatial Fourier transform in Schwartz space, for $\varphi\in S(\mathbb{R}\times\mathbb{T})$, by 
\begin{equation*}
\mathcal{F}(\varphi)(\xi,p)=\hat{\varphi}_p(\xi):=\frac{1}{(2\pi)^2}\displaystyle\int_{\mathbb{R}\times\mathbb{T}}e^{-ix\xi}e^{-iyp}\varphi(x,y)dxdy. 
\end{equation*}
We see with this definition that we use the notation $\hat{f}(\xi)$ for the Fourier transform of a function defined on~$\mathbb{R}$,
and the notation $a_p$ for the Fourier transform of a function defined on the torus $\mathbb{T}$.
An interesting feature of this Fourier transform convention is that there is no $\pi$-coefficient for the inverse Fourier transform.

As mentioned in the motivations, it seems relevant to introduce the profiles $(F,G)$ of a solution $(U,V)$, defined by
\begin{equation*}
F(t,x,y):=e^{-it\Delta_{\mathbb{R}\times\mathbb{T}}}U(t,x,y) \qquad\qquad G(t,x,y):=e^{-it\Delta_{\mathbb{R}\times\mathbb{T}}}V(t,x,y). 
\end{equation*}
As we work with small initial data, we can expect the nonlinearity to stay small, and thus the solutions of the system \eqref{sys} to stay close
to their profiles which are the solutions of the linear system.

Concerning the frequency sets we use here, we first introduce the momentum level set
\begin{align*}
\mathcal{M}:=\{(p,q,r,s)\in\mathbb{Z}^4\;:\;m(p,q,r,s):=p-q+r-s=0\}.
\end{align*}
Thanks to this set, we define the resonant level sets by
\begin{equation*}
 \Gamma_\omega:=\{(p,q,r,s)\in\mathcal{M}\;:\;\omega(p,q,r,s):=p^2-q^2+r^2-s^2=\omega\}.
\end{equation*}
In particular, the resonant level set associated to our resonant system is the set $\Gamma_0$.
Due to the dimension one for the torus, it is straightforward to check that 
\begin{equation*}
 (p,q,r,s)\in\Gamma_0 \quad \text{if and only if} \quad \left\{p,r\right\}=\left\{q,s\right\}.
\end{equation*}

Finally, for the constants used in the paper, we fix the small parameter $\delta\leq 10^{-3}$ and the integer
$N\geq12$ for the definitions of the norms we present just below.

\subsubsection{The norms}
First, for the sequences $a=(a_p)_{p\in\mathbb{Z}}$, we defined the associated Sobolev norm by
\begin{equation*}
 \|a\|^2_{h^s_p}:=\sum_{p\in\mathbb{Z}}(1+p^2)^s|a_p|^2.
\end{equation*}
Then, for the functions, 
we use two strong norms $S$ and $S^+$ defined by
\begin{equation}\label{defS}
\|F\|_{S}:=\|F\|_{H^N_{x,y}}+\|xF\|_{L^2_{x,y}}, \quad \|F\|_{S^+}:=\|F\|_{S}+\|(1-\partial_{xx})^4F\|_{S}+\|xF\|_{S}.
\end{equation}
The subscripted letters $x,y$ and $p$ in the definitions of the norms indicate the variable concerned by the integration, and thus the canonical
integration domain associated ($\mathbb{R}$ for the variables $x$, $y$ and $\xi$; and $\mathbb{Z}$ for the variable $p$).
We remark that the $S^+$ norm is a stronger norm than the $S$ norm, but only in $x$.

\subsection{Introduction of the profiles}\label{subprofi}

By writing the system \eqref{sys} with the profiles (defined in \eqref{defprof}), we get
\begin{equation*}\begin{dcases}
i\partial_tF(t)&=e^{-it\Delta_{\mathbb{R}\times\mathbb{T}}}\left(e^{it\Delta_{\mathbb{R}\times\mathbb{T}}}G(t)e^{-it\Delta_{\mathbb{R}\times\mathbb{T}}}
\overline{G(t)}e^{it\Delta_{\mathbb{R}\times\mathbb{T}}}F(t)\right)=\mathcal{N}^t[G(t),G(t),F(t)],\\ 
i\partial_tG(t)&=e^{-it\Delta_{\mathbb{R}\times\mathbb{T}}}\left(e^{it\Delta_{\mathbb{R}\times\mathbb{T}}}F(t)e^{-it\Delta_{\mathbb{R}\times\mathbb{T}}}
\overline{F(t)}e^{it\Delta_{\mathbb{R}\times\mathbb{T}}}G(t)\right)=\mathcal{N}^t[F(t),F(t),G(t)],
\end{dcases}\end{equation*}
where 
\begin{equation*}
\mathcal{N}^t[F(t),G(t),H(t)]:=e^{-it\Delta_{\mathbb{R}\times\mathbb{T}}}\left(e^{it\Delta_{\mathbb{R}\times\mathbb{T}}}
Fe^{-it\Delta_{\mathbb{R}\times\mathbb{T}}}\overline{G}e^{it\Delta_{\mathbb{R}\times\mathbb{T}}}H\right).\end{equation*}
Let us compute the Fourier transform of $\mathcal{N}^t[F(t),G(t),H(t)]$. We first remark that, taking the Fourier transform with respect to the $y$ variable, 
we get for F
\begin{equation*}
e^{it\Delta_{\mathbb{R}\times\mathbb{T}}}F(t,x,y)=\displaystyle\sum_{q\in\mathbb{Z}}e^{iqy}e^{-itq^2}\left(e^{it\partial_{xx}}F_q(t,x)\right).
\end{equation*}
Thus, by taking this expression for $F$,$G$ and $H$, we obtain
\begin{align*}
\mathcal{F}\mathcal{N}^t[F,G,H](\xi,p)
&=\dfrac{e^{it\xi^2}}{(2\pi)^2}\int_{\mathbb{R}\times\mathbb{T}}e^{-ix\xi}
\sum_{q,r,s\in\mathbb{Z}}e^{-im(p,q,r,s)y}e^{it\omega(p,q,r,s)}\left(e^{it\partial_{xx}}F_qe^{-it\partial_{xx}}\overline{G_r}
e^{it\partial_{xx}}H_s\right)(x)dxdy\nonumber\\
&=\dfrac{e^{it\xi^2}}{2\pi}\sum_{(p,q,r,s)\in\mathcal{M}}e^{it\omega(p,q,r,s)}
\int_{\mathbb{R}}e^{-ix\xi}\left(e^{it\partial_{xx}}F_qe^{-it\partial_{xx}}\overline{G_r}
e^{it\partial_{xx}}H_s\right)(x)dx\nonumber\\
&=\sum_{(p,q,r,s)\in\mathcal{M}}e^{it\omega(p,q,r,s)}\widehat{\mathcal{I}^t[F_q,G_r,H_s]}(\xi),\label{FNI}
\end{align*}
where
\begin{equation*}
\mathcal{I}^t[F_q,G_r,H_s](x):=e^{-it\partial_{xx}}\left(e^{it\partial_{xx}}F_qe^{-it\partial_{xx}}\overline{G_r}
e^{it\partial_{xx}}H_s\right)(x).
\label{defI}\end{equation*}
We compute $\widehat{\mathcal{I}^t[F_q,G_r,H_s]}$ using the properties of the Fourier transform, we have
\begin{align*}
\widehat{\mathcal{I}^t[F_q,G_r,H_s]}(\xi)&=e^{it\xi^2}\widehat{\left(e^{it\partial_{xx}}F_qe^{-it\partial_{xx}}\overline{G_r}
e^{it\partial_{xx}}H_s\right)}(\xi)\\
&=e^{it\xi^2}\int_{\mathbb{R}^2}\widehat{(e^{it\partial_{xx}}F_q)}(\xi-a-b)\widehat{(e^{-it\partial_{xx}}\overline{G_r})}(a)
\widehat{(e^{it\partial_{xx}}H_s)}(b)dadb\\
&=e^{it\xi^2}\int_{\mathbb{R}^2}e^{it\left(b^2-a^2-(\xi-a-b)^2\right)}\hat{F}_q(\xi-a-b)\widehat{\overline{G_r}}(a)\hat{H}_s(b)dadb,
\intertext{we now use the changes of variable $\eta=a+b$ and $\kappa=\xi-a$ to get}
\widehat{\mathcal{I}^t[F_q,G_r,H_s]}(\xi)
&=\int_{\mathbb{R}^2}e^{2it\eta\kappa}\hat{F}_q(\xi-\eta)\overline{\hat{G}_r}(\xi-\eta-\kappa)\hat{H}_s(\xi-\kappa)d\eta d\kappa.
\end{align*}
Thus, back to $\mathcal{N}^t$, we have
\begin{equation*}
\mathcal{F}\mathcal{N}^t[F,G,H](\xi,p)=\sum_{(p,q,r,s)\in\mathcal{M}}e^{it\omega(p,q,r,s)}
\int_{\mathbb{R}^2}e^{2it\eta\kappa}\hat{F}_q(\xi-\eta)\overline{\hat{G}_r}(\xi-\eta-\kappa)\hat{H}_s(\xi-\kappa)d\eta d\kappa.
\end{equation*}
From this equation, by a formal  intuition given by the stationary phase principle and the Birkhoff normal forms, 
we define the resonant part $\mathcal{R}$ of the nonlinearity~$\mathcal{N}^t$ by
\begin{equation}\label{defR}
\mathcal{F}\mathcal{R}[F,G,H](\xi,p)=\sum_{(p,q,r,s)\in\Gamma_0}\hat{F}_q(\xi)\overline{\hat{G}_r}(\xi)\hat{H}_s(\xi).
\end{equation}

\section{The modified scattering and the wave operator}\label{modanswav}

The aim of this section is to prove a modified scattering result and to construct of modified wave operator between the initial system \eqref{sys} 
and the resonant system \eqref{res}.
For this purpose, we first study a decomposition of the nonlinearity $\mathcal{N}^t$ in order to highlight the proximity between $\mathcal{N}^t$ and 
$\frac{\pi}t\mathcal{R}$.
Then, we state here two theorems which are extensions of the modified scattering and wave operator results of Hani, Pausader, Tzvetkov and Visciglia in~\cite{HPTV} 
to our coupled system case.
We refer to their paper for the complete proofs of these theorems, because the pass from the equation to the system doesn't bring new technical difficulties.
The complete proofs in the system case are also available in \cite{VVDRthese}.
In the two theorems, $S^+$ is a strong $L^2_{x,y}$ based Banach space, introduced in Subsection \ref{subnorm}, which contains the Schwartz functions. 
Moreover, we recall that $N\geq12$ is an integer.

\subsection{Decomposition of the nonlinearity}

We recall that the nonlinearity $\mathcal{N}^t$ is given in the Fourier space by
\begin{equation*}
\mathcal{F}\mathcal{N}^t[F,G,H](\xi,p)=\sum_{p-q+r-s=0}e^{it(p^2-q^2+r^2-s^2)}
\int_{\mathbb{R}^2}e^{2it\eta\kappa}\hat{F}_q(\xi-\eta)\overline{\hat{G}_r}(\xi-\eta-\kappa)\hat{H}_s(\xi-\kappa)d\eta d\kappa.
\end{equation*}
For the integral part of the nonlinearity, we use the following stationary phase principle from Zuily in \cite{Zui02}: 
\begin{lem}
Let $n\in\mathbb{N}^{*}$ be an integer and $a\in\mathcal{C}^\infty_0(\mathbb{R}^n,\mathbb{C})$ a smooth function with compact support.
If the phase $\varphi\in\mathcal{C}^\infty(\mathbb{R}^n,\mathbb{R})$ admits an unique critical point $x_c\in\mathbb{R}^n$, non-degenerate, then 
\begin{equation*}
\int_{\mathbb{R}^n}e^{it\varphi(x)}a(x)\mathrm{d}x
\underset{t\rightarrow+\infty}\sim
e^{it\varphi(x_c)}\left(\frac{2\pi}t\right)^{\frac n2}\left|\det \varphi''(x_c)\right|^{-\frac12}
e^{i\frac\pi4\sgn(\varphi''(x_c))}a(x_c),
\end{equation*}
where $\sgn(\varphi''(x_c))$ is the signature of the Hessian of $\varphi$ in $x_c$.
\end{lem}
If our phase $\varphi(\eta,\kappa):=2\eta\kappa$ satisfies the assumptions of the lemma (the origin is the only critical point and is non-degenerate), 
this is not the case for the amplitude $a$.
Nevertheless, the stationary phase principle gives us a candidate to be the predominant term of the integral part, we expect that
\begin{equation*}
\int_{\mathbb{R}^2}e^{2it\eta\kappa}\hat{F}_q(\xi-\eta)\overline{\hat{G}_r}(\xi-\eta-\kappa)\hat{H}_s(\xi-\kappa)d\eta d\kappa
\underset{t\rightarrow+\infty}\sim\frac{\pi}t\hat{F}_q(\xi)\overline{\hat{G}_r(\xi)}\hat{H}_s(\xi).
\end{equation*}
To take the sum into account, we use a Birkhoff normal form intuition (see \cite{GPT}, \cite{GreTho12}, \cite{GreVil11}, \cite{HauPro16}\ldots). 
Therefore, we expect that 
\begin{equation*}
\sum_{p-q+r-s=0}a(t,p,q,r,s)\underset{\text{up to a remainder term}}\sim\sum\limits_{\substack{p-q+r-s=0 \\ p^2-q^2+r^2-s^2=0}}a(t,p,q,r,s).
\end{equation*}
Thanks to the stationary phase and the Birkhoff normal form intuitions, we finally expect that
\begin{equation*}
\mathcal{N}^t[F,G,H]\underset{t\rightarrow+\infty}\sim\frac\pi t\mathcal{R}[F,G,H], 
\end{equation*}
where $\mathcal{R}$ is defined in \eqref{defR}.
Thus, we study the decomposition 
\begin{equation}\label{decom}
 \mathcal{N}^t[F,G,H]=\frac\pi t\mathcal{R}[F,G,H]+\mathcal{E}^t[F,G,H].
\end{equation}
In order to show that the dynamics of the initial system (associated to $\mathcal{N}^t$) are close to the dynamics of the resonant system (associated to 
$\mathcal{R}$), we now have to show that $\mathcal{E}^t$ is a good remainder term (i.e a term that enjoys a good decay).

\subsection{Study of the remainder term \texorpdfstring{$\mathcal{E}^t$}{Et}}
The study of the decay of the remainder term $\mathcal{E}^t$ is the purpose of Proposition 3.1 in \cite{HPTV}.
The idea is to use a frequency decomposition and resonant/nonresonant decomposition of the nonlinearity $\mathcal{N}^t$.

First, thanks to the Littlewood-Paley decomposition, we split the nonlinearity into two parts:
\begin{equation*}
 \mathcal{N}^t=\mathcal{N}^t_{\text{low}}+\mathcal{N}^t_{\text{high}},
\end{equation*}
where the notation $\mathcal{N}^t_{\text{low}}$ (respectively $\mathcal{N}^t_{\text{high}}$) stands for the low (respectively high) frequency part of the 
nonlinearity. The idea here is to show that the term $\mathcal{N}^t_{\text{high}}$ has a good decay, thus it is a part of the remainder term $\mathcal{E}^t$.
For that purpose, we use a classical bilinear Strichartz estimate (see Lemma 7.2 in~\cite{HPTV} for the statement, two different proofs are given by 
Colliander, Keel, Staffilani, Takaoka and Tao in \cite{CKSTT01} and Tao in \cite{taophysical}). In particular, in order to control the high frequency part of 
the nonlinearity, we need the $H^N$ part of the $S$ norm (defined in \eqref{defS}) and the fact that we choose $N$ large enough (in \cite{HPTV}, they choose 
$N\geq30$ to be sure that $N$ is large enough, but we see here that $N\geq12$ is sufficient). 
This is the purpose of Lemma 3.2 in \cite{HPTV}.

Then, we have to study the low frequency term $\mathcal{N}^t_{\text{low}}$. Here, we make a decomposition between resonant (with $p^2-q^2+r^2-s^2=0$) 
and nonresonant (with $p^2-q^2+r^2-s^2\neq0$) parts. We obtain 
\begin{equation*}
 \mathcal{N}^t_{\text{low}}=\mathcal{N}^t_{\text{low,res}}+\tilde{\mathcal{N}}^t_{\text{low}}.
\end{equation*}
Thanks to the term $\|xF\|_{L^2_{x,y}}$ term of the $S$ norm in \eqref{defS}, we prove that the nonresonant part $\tilde{\mathcal{N}}^t_{\text{low}}$ has 
also a good decay and can be integrate to the remainder part $\mathcal{E}^t$ (Lemma 3.3 in \cite{HPTV}). 

Finally, we have to deal with the low frequency resonant part $\mathcal{N}^t_{\text{low,res}}$. For that purpose, we use the proximity between the resonant part 
$\mathcal{N}^t_{\text{res}}$ of $\mathcal{N}^t$ and the resonant nonlinearity $\frac{\pi}t\mathcal{R}$. This proximity is given by Lemma 3.7 in \cite{HPTV}.
For our system case, the proof is simpler due to the dimension one for the torus (see Proposition 4.3.9 in \cite{VVDRthese}). Therefore, we study the 
decomposition
\begin{equation*}
 \mathcal{N}^t_{\text{low,res}}=\frac\pi t\mathcal{R}+\left(\mathcal{N}^t_{\text{low,res}}-\frac\pi t\mathcal{R}_{\text{low}}\right)
 -\frac\pi t\mathcal{R}_{\text{high}}.
\end{equation*}
In particular, the last term $\frac\pi t\mathcal{R}_{\text{high}}$ justifies the role of the $\|(1-\partial_{xx})^4F\|_{S}$ term 
in the definition of the $S^+$ norm in \eqref{defS} (see Proposition 4.3.15 in \cite{VVDRthese}). The remainder term in \eqref{decom} is thus defined by
\begin{equation*}
 \mathcal{E}^t=\mathcal{N}^t_{\text{high}}+\tilde{\mathcal{N}}^t_{\text{low}}
 +\left(\mathcal{N}^t_{\text{low,res}}-\frac\pi t\mathcal{R}_{\text{low}}\right)-\frac\pi t\mathcal{R}_{\text{high}}.
\end{equation*}

\subsection{The modified scattering}
First, the following theorem shows that solutions of the initial system~\eqref{sys} scatter to solutions of the resonant system \eqref{res}.

\begin{thm}\label{modscat}
There exists $\varepsilon>0$ such that if $U_0,V_0\in S^+$ satisfy
\begin{equation*}
\|U_0\|_{S^+}+\|V_0\|_{S^+}\leq\varepsilon,
\end{equation*}
and if $(U(t),V(t))$ solves \eqref{sys} with initial data $(U_0,V_0)$, then $(U,V)\in\mathcal{C}([0,+\infty):H^N)\times\mathcal{C}([0,+\infty) : H^N)$ exists 
globally and exhibits modified scattering to its resonant dynamics \eqref{res} in the following sense: there exists
$(W_{U,0},W_{V,0})$ satisfying
\begin{equation}\label{wu0control}
\|W_{U,0}\|_{S}+\|W_{V,0}\|_{S}\lesssim\varepsilon,
\end{equation}
such that if $(W_{U}(t),W_{V}(t))$ is the solution of \eqref{res} with initial data $(W_{U,0},W_{V,0})$, then
\begin{equation*}\begin{dcases}                               
\|U(t)-e^{it\Delta_{\mathbb{R}\times\mathbb{T}}}W_U(\pi\ln(t))\|_{H^N(\mathbb{R}\times\mathbb{T})}&\rightarrow0 \quad \text{as} \quad t\rightarrow+\infty,\\
\|V(t)-e^{it\Delta_{\mathbb{R}\times\mathbb{T}}}W_V(\pi\ln(t))\|_{H^N(\mathbb{R}\times\mathbb{T})}&\rightarrow0 \quad \text{as} \quad t\rightarrow+\infty.
\end{dcases}
\end{equation*}
Furthermore, we have the following decay estimate
\begin{equation*}\begin{dcases}                               
\|U(t)\|_{L^\infty_xH^1_y}&\lesssim(1+|t|)^{-\frac12},\\
\|V(t)\|_{L^\infty_xH^1_y}&\lesssim(1+|t|)^{-\frac12}.
\end{dcases}
\end{equation*}
\end{thm}

\begin{proof}[Idea of the proof]
First, by classical tools, we prove that the solution of the initial system \eqref{sys} exists globally. 
Using the direction of dispersion, the decay estimate is also standard. 
For the scattering result, we take a solution $(U,V)$ of \eqref{sys} which satisfies, for $\varepsilon$ small enough,
\begin{equation*}
\|U_0\|_{S^+}+\|V_0\|_{S^+}\leq\varepsilon.
\end{equation*}
Then, we define
\begin{equation*}
 W_U^n(t):=\tilde{W}^n_U(\pi\ln(t)) \quad \text{and} \quad W_V^n(t):=\tilde{W}_V^n(\pi\ln(t)),
\end{equation*}
where $(\tilde{W}^n_U,\tilde{W}_V^n)$ is the solution of the resonant system \eqref{res} with Cauchy data
\begin{equation*}
\tilde{W}^n_U(n)=W_U^n(e^{\frac n\pi})=F(e^{\frac n\pi}), \qquad \tilde{W}^n_V(n)=W_V^n(e^{\frac n\pi})=G(e^{\frac n\pi}), 
\end{equation*}
with the profiles $F$ and $G$ of $U$ and $V$ defined in \eqref{defprof}.
In order to prove the scattering result, we want to show that the couple of profiles $(F,G)$ tends to be close to a couple of solutions of the resonant 
system~\eqref{res}.
By definition of $W_U^n$ and $W_V^n$, we have 
\begin{equation*}
 i\partial_t W^n_U(t)=\frac{\pi}t\mathcal{R}[W_V^n(t),W_V^n(t),W_U^n(t)], \quad i\partial_t W_V^n(t)=\frac{\pi}t\mathcal{R}[W_U^n(t),W_U^n(t),W_V^n(t)].
\end{equation*} 
Thus, using the decomposition \eqref{decom} and the choice of initial data, we have
\begin{align*}
F(t)-W_U^n(t)&=-i\int_{e^{\frac n\pi}}^t (\mathcal{N}^s[G,G,F]-\frac{\pi}s\mathcal{R}[W_V^n(s),W_V^n(s),W_U^n(s)])\mathrm{d}s\\
\qquad&=-i\int_{e^{\frac n\pi}}^t \mathcal{E}^s[G,G,F]ds-i\int_{e^{\frac n\pi}}^t \frac{\pi}s
\big(\mathcal{R}[G(s),G(s),F(s)]-\mathcal{R}[W_V^n(s),W_V^n(s),W_U^n(s)]\big)\mathrm{d}s.
\end{align*}
Thanks to the good decay of the remainder term $\mathcal{E}^t$ and the control of the resonant $\mathcal{R}$, and using the Gr\"onwall's Lemma, we show that
\begin{equation*}
 \sup_{e^{\frac n\pi}\leq t\leq e^{\frac {(n+4)}\pi}}\|F(t)-W_U^n(t)\|_S\underset{n\to +\infty}{\longrightarrow} 0.
\end{equation*}
We see here that $F(t)$ is close to $\tilde{W}_U^n(\pi\ln(t))$. 
However, we can't conclude that one $(\tilde{W}_U^n,\tilde{W}_V^n)$ couple is solution of the problem, because if we want to take $t\rightarrow+\infty$, 
the time $t$ becomes bigger than any $e^{\frac {(n+4)}\pi}$ for a fixed $n$. The idea is thus to find some limit of the $\tilde{W}_U^n$ terms. 
Because of the different intervals $[e^{\frac n\pi},e^{\frac {(n+4)}\pi}]$ for the estimates, we can't hope to take directly the limit of the $\tilde{W}_U^n$. 
This is the reason why we look now at the initial data. 
The idea is to show that the sequences $(\tilde{W}_U^{n}(0))_n$ and $(\tilde{W}_V^{n}(0))_n$ are Cauchy sequences in the complete space $S$ 
(defined by the S norm). 
Thus, there exists a couple of elements $(W_{U,0,\infty},W_{V,0,\infty})\in S\times S$, limit of these sequences. 
We denote now by $(\tilde{W}_{U,\infty},\tilde{W}_{V,\infty})$ the couple of solutions of the resonant system \eqref{res} with the couple of initial data 
$(W_{U,0,\infty},W_{V,0,\infty})$. 
Finally, the couple $(\tilde{W}_{U,\infty},\tilde{W}_{V,\infty})$ is the solution of the problem in the sense where we can show that (the same relation is 
satisfied \textit{mutatis mutandis} for V):
\begin{align*}
\|e^{-\Delta_{\mathbb{R}\times\mathbb{T}}}U(t)-\tilde{W}_{U,\infty}(\pi\ln(t))\|_S=\|F(t)-\tilde{W}_{U,\infty}(\pi\ln(t))\|_S
\underset{t\to +\infty}{\longrightarrow}0.
\end{align*}
In particular, the relation is true for the $H^N$ norm which is a part of the $S$ norm (see definition \eqref{defS}), and the proof of the scattering result 
is completed.
\end{proof}

\subsection{The modified wave operator}

After the existence of a scattering result, the natural question is to look for a wave operator. Indeed, if we know that each couple of solutions of the system
\eqref{sys} scatters to a couple of solutions of the resonant system \eqref{res}, we want to know if all the couples of solutions of the resonant system \eqref{res} 
are limits of couples of solutions of the initial system. The answer is given by the following theorem:

\begin{thm}\label{waveop}
There exists $\varepsilon>0$ such that if $W_{U,0},W_{V,0}\in S^+$ satisfy
\begin{equation*}
\|W_{U,0}\|_{S^+}+\|W_{V,0}\|_{S^+}\leq\varepsilon, 
\end{equation*}
and if $(W_{U}(t),W_{V}(t))$ solves the resonant system \eqref{res} with initial data $(W_{U,0},W_{V,0})$, then there exists a couple
$(U,V)\in\mathcal{C}([0,+\infty):H^N)\times\mathcal{C}([0,+\infty) : H^N)$ solution of \eqref{sys} such that
\begin{equation*}\begin{dcases}                               
\|U(t)-e^{it\Delta_{\mathbb{R}\times\mathbb{T}}}W_U(\pi\ln(t))\|_{H^N(\mathbb{R}\times\mathbb{T})}&\rightarrow0 \quad \text{as} \quad t\rightarrow+\infty,\\
\|V(t)-e^{it\Delta_{\mathbb{R}\times\mathbb{T}}}W_V(\pi\ln(t))\|_{H^N(\mathbb{R}\times\mathbb{T})}&\rightarrow0 \quad \text{as} \quad t\rightarrow+\infty.
\end{dcases}
\end{equation*}
\end{thm}

\begin{proof}[Idea of the proof]
The idea is to use a fixed-point theorem. First, as in the previous theorem, we perform a change of variable by setting 
$\tilde{W}_U(t)=W_U(\pi\ln(t))$ and $\tilde{W}_V(t)=W_V(\pi\ln(t))$ in order to obtain
\begin{equation*}
 i\partial_t \tilde{W}_U(t)=\frac{\pi}t\mathcal{R}[\tilde{W}_V(t),\tilde{W}_V(t),\tilde{W}_U(t)], 
 \quad i\partial_t \tilde{W}_V(t)=\frac{\pi}t\mathcal{R}[\tilde{W}_U(t),\tilde{W}_U(t),\tilde{W}_V(t)].
\end{equation*}
The map we define for the fixed-point theorem is 
\begin{equation*}
 \Phi(F,G)(t):=(\Phi_1(F,G)(t),\Phi_2(F,G)(t)),
\end{equation*}
where
\begin{align*}
 \Phi_1(F,G)(t):=i\int_t^{+\infty}\left(\mathcal{N}^s[G+\tilde{W}_V,G+\tilde{W}_V,F+\tilde{W}_U]
 -\frac\pi s\mathcal{R}[\tilde{W}_V(s),\tilde{W}_V(s),\tilde{W}_U(s)]\right)\mathrm{d}s,\\
 \Phi_2(F,G)(t):=i\int_t^{+\infty}\left(\mathcal{N}^s[F+\tilde{W}_U,F+\tilde{W}_U,G+\tilde{W}_V]
 -\frac\pi s\mathcal{R}[\tilde{W}_U(s),\tilde{W}_U(s),\tilde{W}_V(s)]\right)\mathrm{d}s.
\end{align*}
The space we define for the fixed-point theorem is the small ball
\begin{equation*}
B_{\tilde{\mathfrak{A}}}(2\varepsilon)=\left\{(F,G)\in\tilde{\mathfrak{A}},\|(F,G)\|_{\tilde{\mathfrak{A}}}\leq2\varepsilon \right\},
\end{equation*}
where 
\begin{equation*}
 \|(F,G)\|_{\tilde{\mathfrak{A}}}=\|F\|_{\mathfrak{A}}+\|G\|_{\mathfrak{A}},
\end{equation*}
and the space $\mathfrak{A}$ is chosen (see the proof of Theorem 5.1 in \cite{HPTV} for the formal definition of the norm) such that 
\begin{equation*}
 \|F\|_{\mathfrak{A}}\leq 2\varepsilon \Rightarrow \|F(t)\|_{S}\underset{t\to +\infty}{\longrightarrow}0.
\end{equation*}
To apply the fixed-point theorem, we need to show that for $\varepsilon$ small enough, $\Phi$ is a contraction (see the proof of Theorem 4.6.1 in 
\cite{VVDRthese}). 
Let us see how it is sufficient to prove the theorem. 
Assume that $\Phi$ is a contraction, by the fixed-point theorem, there exists an unique couple $(F,G)$ such that $\Phi(F,G)=(F,G)$. It follows that
\begin{align*}
 i\partial_tF(t)=\mathcal{N}^t[G+\tilde{W}_V,G+\tilde{W}_V,F+\tilde{W}_U]-\frac\pi t\mathcal{R}[\tilde{W}_V(t),\tilde{W}_V(t),\tilde{W}_U(t)],\\
 i\partial_tG(t)=\mathcal{N}^t[F+\tilde{W}_U,F+\tilde{W}_U,G+\tilde{W}_V]-\frac\pi t\mathcal{R}[\tilde{W}_U(t),\tilde{W}_U(t),\tilde{W}_V(t)].
\end{align*}
Thus, with the choice of $(\tilde{W}_U,\tilde{W}_V)$, we get
\begin{align*}\begin{dcases}
 i\partial_t(F+\tilde{W}_U)(t)=\mathcal{N}^t[G+\tilde{W}_V,G+\tilde{W}_V,F+\tilde{W}_U],\\
 i\partial_t(G+\tilde{W}_V)(t)=\mathcal{N}^t[F+\tilde{W}_U,F+\tilde{W}_U,G+\tilde{W}_V].
\end{dcases}\end{align*}
Now we set $U:=e^{it\Delta_{\mathbb{R}\times\mathbb{T}}}(F+\tilde{W}_U)$ and $V:=e^{it\Delta_{\mathbb{R}\times\mathbb{T}}}(G+\tilde{W}_V)$, 
then $(U,V)$ solves \eqref{sys} and by definition, for $(F,G)\in B_{\tilde{\mathfrak{A}}}(2\varepsilon)$, we have
\begin{align*}
 \|e^{-it\Delta_{\mathbb{R}\times\mathbb{T}}}U(t)-W_U(\pi\ln(t))\|_S=\|F(t)\|_S \underset{t\to +\infty}{\longrightarrow} 0,\\
 \|e^{-it\Delta_{\mathbb{R}\times\mathbb{T}}}V(t)-W_V(\pi\ln(t))\|_S=\|G(t)\|_S \underset{t\to +\infty}{\longrightarrow} 0.
\end{align*}
The proof of the theorem is thus completed.
\end{proof}

The aim now is to use these two theorems to obtain some dynamical consequences for the initial system~\eqref{sys}.

\section{Study of the resonant system}\label{Secres}
In this section, we want to study the resonant system 
\begin{equation*}\begin{dcases}
i\partial_{\tau}W_U(\tau)=\mathcal{R}\left[W_V(\tau),W_V(\tau),W_U(\tau)\right],\\
i\partial_{\tau}W_V(\tau)=\mathcal{R}\left[W_U(\tau),W_U(\tau),W_V(\tau)\right],
\end{dcases}\end{equation*}
with $\mathcal{F}\mathcal{R}\left[F,G,H\right](\xi,p)=\displaystyle\sum_{(p,q,r,s)\in\Gamma_0}
\hat{F}(\xi,q)\hat{G}(\xi,r)\hat{H}(\xi,s)$ and $\Gamma_0=\left\{(p,q,r,s)\in\mathbb{Z}^4, \left\{p,r\right\}=\left\{q,s\right\} \right\}$.
In the definition of $\mathcal{R}$, the Euclidean variable $\xi$ acts just like a parameter. According to this idea, we define the reduced resonant
system for two vectors $a=\left\{a_p\right\}_{p\in\mathbb{Z}}$ and $b=\left\{b_p\right\}_{p\in\mathbb{Z}}$, by
\begin{equation}\begin{dcases}\label{redu}
i\partial_ta(t)=R\left(b(t),b(t),a(t)\right),\\
i\partial_tb(t)=R\left(a(t),a(t),b(t)\right),
\end{dcases}\end{equation}
with $R\left(a(t),b(t),c(t)\right)_p=\displaystyle\sum_{(p,q,r,s)\in\Gamma_0}a_q(t)\overline{b_r(t)}c_s(t).$

First, we study the behavior of the resonant system in Subsection \ref{ss40} in order to obtain the local existence of the solutions of the resonant system.
Then, in Subsection \ref{ss41}, we prove the control of the Sobolev norms of the solutions of system \eqref{sys} and the global existence of the solutions. 
In order to take advantage of the reduced resonant system, we want to show in Subsection~\ref{ss42} that we can transfer solutions from the reduced resonant system 
associated to $R$ to solutions of the resonant system associated to $\mathcal{R}$. Thus, in Subsection \ref{ss43}, we study the structure of this reduced resonant 
system. We obtain in Subsection \ref{ss45} an example of nonlinear behavior with the beating effect. 
It allows us to prove Theorem \ref{beating} in Subsection \ref{ss46}. Finally, we study the persistence of the beating effect through the add of convolution 
potentials in Subsection \ref{ss47}.

\subsection{Behavior of the resonant part}\label{ss40}

The first lemma we need concerns the resonant part $\mathcal{R}$, we have

\begin{lem}\label{lemR}
 For every sequences $(a^1)_p$, $(a^2)_p$ and $(a^3)_p$ indexed by $\mathbb{Z}$, we have
\begin{equation}\label{eqlemR1}
 \|\sum_{(p,q,r,s)\in\Gamma_0}a^1_q\overline{a}^2_ra^3_s\|_{\ell^2_p}\lesssim\min_{\sigma\in\mathfrak{S}_3}\|a^{\sigma(1)}\|_{\ell^2_p}\|a^{\sigma(2)}\|_{h^1_p}
 \|a^{\sigma(3)}\|_{h^1_p}.
\end{equation}
More generally, for all $\nu\geq0$, it holds
\begin{equation}\label{eqlemR2}
 \|\sum_{(p,q,r,s)\in\Gamma_0}a^1_q\overline{a}^2_ra^3_s\|_{h^\nu_p}\lesssim\sum_{\sigma\in\mathfrak{S}_3}\|a^{\sigma(1)}\|_{h^\nu_p}\|a^{\sigma(2)}\|_{h^1_p}
 \|a^{\sigma(3)}\|_{h^1_p}.
\end{equation}
\end{lem}

\begin{rem}
The operator we deal with here is not really the resonant nonlinearity $\mathcal{R}$, because we don't deal with the $\xi$ variable.
In fact, this operator is the reduced resonant nonlinearity we introduce later in Subsection \ref{ss43}. 
\end{rem}

\begin{proof}
We proceed by duality. Set the sequence $(R)_p=(\sum_{(p,q,r,s)\in\Gamma_0}a^1_q\overline{a}^2_ra^3_s)_p$. For $a^0\in {\ell^2_p}$ we have
\begin{equation*}
\left<a^0,R\right>_{\ell^2_p\times \ell^2_p}=\sum_{(p,q,r,s)\in\Gamma_0}a^0_p\overline{a}^1_qa^2_r\overline{a}^3_s.
\end{equation*}
Once again, we use that in one dimension, $(p,q,r,s)\in\Gamma_0$ implies $\{p,r\}=\{q,s\}$. Thus
\begin{equation*}
\left<a^0,R\right>_{\ell^2_p\times \ell^2_p}=\sum_{p,r\in\mathbb{Z}}a^0_p\overline{a}^1_pa^2_r\overline{a}^3_r
+\sum_{p,r\in\mathbb{Z}}a^0_p\overline{a}^1_ra^2_r\overline{a}^3_p-\sum_{p\in\mathbb{Z}}a^0_p\overline{a}^1_pa^2_p\overline{a}^3_p.
\end{equation*}
For the first term, we have by Cauchy-Schwarz in $p$ and $r$
\begin{equation*}
 \sum_{p,r\in\mathbb{Z}}a^0_p\overline{a}^1_pa^2_r\overline{a}^3_r=\left(\sum_{p\in\mathbb{Z}}a^0_p\overline{a}^1_p\right)
 \left(\sum_{r\in\mathbb{Z}}a^2_r\overline{a}^3_r\right)\leq \|a^0\|_{\ell^2_p}\|a^1\|_{\ell^2_p}\|a^2\|_{\ell^2_p}\|a^3\|_{\ell^2_p}.
\end{equation*}
The same holds for the second term. For the last term, we have by Cauchy Schwarz
\begin{equation*}
\sum_{p\in\mathbb{Z}}a^0_p\overline{a}^1_pa^2_p\overline{a}^3_p\leq\|a^0\|_{\ell^2_p}
\min_{\sigma\in\mathfrak{S}_3}\|a^{\sigma(1)}\|_{\ell^2_p}\|a^{\sigma(2)}\|_{\ell^\infty_p}\|a^{\sigma(3)}\|_{\ell^\infty_p},
\end{equation*}
which is sufficient due to the embedding $h^1_p\hookrightarrow\ell^2_p\hookrightarrow\ell^\infty_p$. Therefore, we get 
\begin{equation*}
\left<a^0,R\right>_{\ell^2_p\times \ell^2_p}\lesssim\|a^0\|_{\ell^2_p}
\min_{\sigma\in\mathfrak{S}_3}\|a^{\sigma(1)}\|_{\ell^2_p}\|a^{\sigma(2)}\|_{h^1_p}\|a^{\sigma(3)}\|_{h^1_p}.
\end{equation*}
This completes the proof of \eqref{eqlemR1}.
For equation \eqref{eqlemR2}, we remark that 
\begin{equation*}
 \text{for} \, (p,q,r,s)\in\Gamma_0, \qquad p^2-q^2+r^2-s^2=0\Rightarrow\left<p\right>^\nu\leq\left<q\right>^\nu+\left<r\right>^\nu+\left<s\right>^\nu.
\end{equation*}
In fact, here we have $\left<p\right>=\left<q\right>$ or $\left<s\right>$, but the previous inequality is sufficient for the estimate. Indeed, this 
inequality implies 
\begin{equation*}
\left<a^0,\left<p\right>^\nu R\right>_{\ell^2_p\times \ell^2_p}\leq
\sum_{(p,q,r,s)\in\Gamma_0}|a^0_p|\left(\left<q\right>^\nu|\overline{a}^1_qa^2_r\overline{a}^3_s|
+|\overline{a}^1_q|\left<r\right>^\nu |a^2_r\overline{a}^3_s|
+|\overline{a}^1_qa^2_r|\left<s\right>^\nu|\overline{a}^3_s|\right).
\end{equation*}
Then, it suffices to apply \eqref{eqlemR1} to each part of the sum, taking the $\ell^2_p$ for the term in the sum with the weight.
This concludes the proof of this lemma.
\end{proof}

\begin{rem}
 The fact that we take $p,q,r,s\in\mathbb{Z}$ is crucial for this proof. In higher dimension (in $\mathbb{R}\times\mathbb{T}^d$ with $2\leq d\leq4$), 
 this method doesn't fit anymore, whereas the result always holds.
 Taking $(p,q,r,s)\in\Gamma_0$ implies that $p,q,r,s$ are the vertices of a rectangle. Thus, in dimension one, we have a flat rectangle
 and $\Gamma_0$ is a trivial set. In dimension 2 or more, $\Gamma_0$ becomes much harder, and the proof of Lemma \ref{lemR} 
 becomes more technical, as we can see in \cite[Lemma 7.1]{HPTV}.
\end{rem}

From this lemma, we deduce the local existence of the solutions of the resonant system \eqref{res} through the following corollary. The global existence is given 
in the next subsection by Lemma \ref{globres}.

\begin{cor}\label{coroloc}
Let $\nu\in\mathbb{N}^*$. For any $(W_{U,0},W_{V,0})\in H^\nu_{x,y}\times H^\nu_{x,y}$, there exist $T>0$ and a couple 
$(W_{U},W_{V})\in\mathcal{C}([0,T):H^\nu_{x,y})\times\mathcal{C}([0,T):H^\nu_{x,y})$ of local solutions to the resonant system \eqref{res}
with initial data $(W_{U,0},W_{V,0})$.
\end{cor}

\begin{proof}
From the previous lemma, we have for three functions $F^1,F^2$ and $F^3$ on $\mathbb{R}\times\mathbb{T}$,  
\begin{equation*}
 \|\sum_{(p,q,r,s)\in\Gamma_0}F^1_q\overline{F}^2_rF^3_s\|_{h^\nu_p}\lesssim\sum_{\sigma\in\mathfrak{S}_3}\|F^{\sigma(1)}\|_{h^\nu_p}\|F^{\sigma(2)}\|_{h^1_p}
 \|F^{\sigma(3)}\|_{h^1_p}.
\end{equation*}
Thus, we obtain
\begin{equation*}
 \|\mathcal{R}[F^1,F^2,F^3]\|_{H^\nu_{x,y}}\lesssim\sum_{\sigma\in\mathfrak{S}_3}\|F^{\sigma(1)}\|_{H^\nu_{x,y}}\|F^{\sigma(2)}\|_{L^\infty_xh^1_p}
 \|F^{\sigma(3)}\|_{L^\infty_xh^1_p}.
\end{equation*}
The embedding $H^1_x \hookrightarrow L^\infty_x$ allows us to conclude that 
\begin{equation*}
 \|\mathcal{R}[F^1,F^2,F^3]\|_{H^\nu_{x,y}}\lesssim\|F^1\|_{H^\nu_{x,y}}\|F^2\|_{H^\nu_{x,y}}\|F^3\|_{H^\nu_{x,y}}.
\end{equation*}
This implies the local existence of the solutions.
\end{proof}

\subsection{Control of the Sobolev norms}\label{ss41}
In this subsection, we study the behavior of solutions of the resonant system \eqref{res} in order to obtain dynamical consequences for the initial 
system \eqref{sys}. The following lemma ensures the control of the Sobolev norms and the global existence of solutions of resonant system \eqref{res}.

\begin{lem}\label{lemW}
Assume $W_{U,0},W_{V,0}\in S^{(+)}$, and let $(W_{U},W_{V})$ be the solution of the resonant system~\eqref{res} with initial data $(W_{U,0},W_{V,0})$.
Then, for $t\geq1$, we have 
\begin{align}
 \|W_{U}(t)\|_{H^\sigma_{x,y}}+\|W_{V}(t)\|_{H^\sigma_{x,y}}&=\|W_{U,0}\|_{H^\sigma_{x,y}}+\|W_{V,0}\|_{H^\sigma_{x,y}},\quad \forall 
 \sigma\in\mathbb{R}.\label{conshs}
\end{align}
\label{lemres}\end{lem}

\begin{proof}
Let us see the computations for $W_{U}$. By definition of the resonant system, we have
\begin{equation*}
i\partial_{t}\mathcal{F}W_U(t,\xi,p)=\mathcal{F}\mathcal{R}\left[W_V,W_V,W_U\right](t,\xi,p)=\displaystyle\sum_{(p,q,r,s)\in\Gamma_0}
\hat{W}_{V,q}(t,\xi)\overline{\hat{W}_{V,r}(t,\xi)}\hat{W}_{U,s}(t,\xi).
\end{equation*}
Thus, for $h:\mathbb{R}\rightarrow\mathbb{R}$ a real function, we have 
\begin{align*}
 \partial_t\sum_{p\in\mathbb{Z}}h(p)|\hat{W}_{U,p}|^2&=2\sum_{p\in\mathbb{Z}}h(p)\Re\left(\partial_t\hat{W}_{U,p}.\overline{\hat{W}_{U,p}}\right)\\
 &=2\sum_{p\in\mathbb{Z}}h(p)\Im\left(i\partial_t\hat{W}_{U,p}.\overline{\hat{W}_{U,p}}\right)\\
 &=2\sum_{(p,q,r,s)\in\Gamma_0}h(p)\Im\left(\overline{\hat{W}_{U,p}}\hat{W}_{V,q}\overline{\hat{W}_{V,r}}\hat{W}_{U,s}\right).
\end{align*}
Let us rewrite the right-hand side by developing the imaginary part. We have 
\begin{equation*}
  \partial_t\sum_{p\in\mathbb{Z}}h(p)|\hat{W}_{U,p}|^2=-i\sum_{(p,q,r,s)\in\Gamma_0}\left(h(p)\overline{\hat{W}_{U,p}}\hat{W}_{V,q}\overline{\hat{W}_{V,r}}
  \hat{W}_{U,s}-h(p)\hat{W}_{U,p}\overline{\hat{W}_{V,q}}\hat{W}_{V,r}\overline{\hat{W}_{U,s}}\right).
\end{equation*}
By symmetry this equation becomes 
\begin{equation*}
  \partial_t\sum_{p\in\mathbb{Z}}h(p)|\hat{W}_{U,p}|^2=-i\sum_{(p,q,r,s)\in\Gamma_0}\left(h(p)-h(s)\right)
  \overline{\hat{W}_{U,p}}\hat{W}_{V,q}\overline{\hat{W}_{V,r}}\hat{W}_{U,s}.
\end{equation*}
Here, we have no more symmetry because of the different $\hat{W}_U$ and $\hat{W}_V$ terms. 
To avoid this problem, we take advantage of the coupling effect by looking the sum of the norms of the solutions. We have 
\begin{align*}
  \partial_t\sum_{p\in\mathbb{Z}}h(p)(|\hat{W}_{U,p}|^2+|\hat{W}_{V,p}|^2)&=-i\sum_{(p,q,r,s)\in\Gamma_0}
  \left(h(p)-h(s)\right)
  \overline{\hat{W}_{U,p}}\hat{W}_{V,q}\overline{\hat{W}_{V,r}}\hat{W}_{U,s}\\
  &\quad-i\sum_{(p,q,r,s)\in\Gamma_0}
  \left(h(p)-h(s)\right)
  \overline{\hat{W}_{V,p}}\hat{W}_{U,q}\overline{\hat{W}_{U,r}}\hat{W}_{V,s}.
\end{align*}
Now, by symmetry we have 
\begin{equation*}
  \partial_t\sum_{p\in\mathbb{Z}}h(p)(|\hat{W}_{U,p}|^2+|\hat{W}_{V,p}|^2)=-i\sum_{(p,q,r,s)\in\Gamma_0}(h(p)-h(q)+h(r)-h(s))
  \overline{\hat{W}_{U,p}}\hat{W}_{V,q}\overline{\hat{W}_{V,r}}\hat{W}_{U,s}.
\end{equation*}
By the structure of set $\Gamma_0$, we have $\left\{p,r\right\}=\left\{q,s\right\}$, thus
\begin{equation}
 \partial_t\sum_{p\in\mathbb{Z}}h(p)(|\hat{W}_{U,p}|^2+|\hat{W}_{V,p}|^2)=0.\label{consh1wu}
\end{equation}
Taking $h(p)=\left<p\right>^{2\sigma}$, the proof of the lemma follows immediately.
\end{proof}

From this lemma, we have a first dynamical consequence (Theorem \ref{thmbounded}), given by the following remark:

\begin{rem}\label{rembounded}
This estimate is related to the dimension one of the compact part $\mathbb{T}$ of the spatial domain.
In particular, due to the modified scattering of Theorem \ref{modscat}, this estimate prevents every kind of growth of the Sobolev norms of the solutions of the 
initial system \eqref{sys}. We have thus a first dynamical consequence of this theorem: all the couple of solutions of the system \eqref{sys} are bounded in 
every Sobolev space $H^s$ (for all $s\in\mathbb{R}$). In particular, we have proved the estimate \eqref{eqboundhs}.
The estimate \eqref{staysmall} comes from equation~\eqref{wu0control} and Theorem \ref{modscat}, this completes the proof of Theorem \ref{thmbounded}.
\end{rem}

Another consequence of this lemma is given by the global existence of the solutions of the resonant system. Indeed, we have
\begin{lem}\label{globres}
Let $\nu\in\mathbb{N^*}$. For any $(W_{U,0},W_{V,0})\in H^\nu_{x,y}\times H^\nu_{x,y}$, there exists a unique couple of solutions
$(W_{U},W_{V})\in\mathcal{C}(\mathbb{R}:H^\nu_{x,y})\times\mathcal{C}(\mathbb{R}:H^\nu_{x,y})$ of the resonant system \eqref{res}
with initial data $(W_{U,0},W_{V,0})$.
\end{lem}

\begin{proof}
The local existence of the solutions is given by Corollary \ref{coroloc}. Then, equation \eqref{conshs} allows us to pass
from local existence to global existence.
\end{proof}

\subsection{From the reduced resonant system to the resonant system}\label{ss42}

As mentioned in the introduction of the section, we want to take profit of the study of the reduced resonant system \eqref{redu}.
In view of this idea, the following computations show how to transfer informations from a solution of the reduced resonant system \eqref{redu}
to a solution of the system \eqref{res}. Therefore, the wave operator theorem allows us to transfer these informations to solutions of the initial
system \eqref{sys}.

Let $(a,b)$ be a couple of solutions of the reduced resonant system \eqref{redu}.
The idea is to take an initial data of the form
\begin{equation*}
 W_{U,0}(x,y)=\check{\varphi}(x)\alpha(y),\qquad \qquad W_{V,0}(x,y)=\check{\varphi}(x)\beta(y),
\end{equation*}
where $\alpha_p=a_p(0)$, $\beta_p=b_p(0)$, $\varphi\in\mathcal{S}(\mathbb{R})$ and $\check{\varphi}$ is the inverse Fourier transform of $\varphi$.
Thus, thanks to the separated variables $x$ and $y$ in the initial data, it's straightforward to check that the solution of the resonant system \eqref{res}
with initial data $(W_{U,0},W_{V,0})$ is given by 
\begin{equation*}
 \mathcal{F}(W_{U})(t,\xi,p)=\varphi(\xi)a_p\left(\varphi(\xi)^2t\right),\qquad \qquad \mathcal{F}(W_{V})(t,\xi,p)=\varphi(\xi)b_p\left(\varphi(\xi)^2t\right).
\end{equation*}
Indeed, we have for example for $W_{U}$:
\begin{align*}
i\partial_t \mathcal{F}(W_{U})(t,\xi,p)=\varphi(\xi)^3(i\partial_ta_p)\left(\varphi(\xi)^2t\right)=
\mathcal{F}\mathcal{R}[W_{V}(t),W_{V}(t),W_{U}(t)]_p(\xi).
\end{align*}
In particular, if $\varphi=1$ on an open interval $I$, then $\mathcal{F}W_{U}(t,\xi,p)=a_p(t)$ and $\mathcal{F}W_{V}(t,\xi,p)=b_p(t)$, for all $t\in\mathbb{R}$
and all $\xi\in I$. Thus, for $\xi\in I$, the constructed solution $(W_U,W_V)$ of the resonant system \eqref{res} behaves like the initial solution $(a,b)$ 
of the reduced resonant system~\eqref{redu}.

\begin{rem}
The solutions we obtain from this method for the resonant system \eqref{res} are constant with respect to $\xi$ on the interval $I$ we choose.
The idea is to take the interval $I$ as big as we want. Thus, we have a big interval in which we conserve the behavior of the resonant system on the torus.
One can think that we completely kill the role of the Euclidean variable $x$ with this construction, whereas the goal of this construction is to take profit of
the dynamics of the resonant system and to gain a large time behavior thanks to the Euclidean variable. Therefore, this method is really adapted to our aim which
is to construct a couple of solution of the initial system \eqref{sys} with asymptotic dynamical properties of the reduced resonant system.
\end{rem}

\subsection{About the reduced resonant system, the Hamiltonian formalism}\label{ss43}
Let us now see some properties of the reduced resonant system.

First, we can deduce some results for this system from the computations of Subsection \ref{ss41}.
From Lemma~\ref{globres}, we deduce the global existence of the solutions of the reduced resonant system:
\begin{lem}\label{globred}
Let $\sigma\in\mathbb{R}$. For any $(a_0,b_0)\in h^\sigma_{p}\times h^\sigma_{p}$, there exists a unique couple of solutions
$(a,b)\in\mathcal{C}(\mathbb{R}:h^\sigma_{p})\times\mathcal{C}(\mathbb{R}:h^\sigma_{p})$ of the resonant system \eqref{redu}
with initial data $(a_0,b_0)$.
\end{lem}
The computations of estimate \eqref{consh1wu} imply that $R[b,b,.]$ is self-adjoint and satisfy
\begin{equation}\label{selfadR}
 \left<iR[b,b,a],a\right>_{h^1_p\times h^1_p}=0, \qquad \forall a,b\in h^1_p.
\end{equation}
More generally, the computations of \eqref{conshs} give us, for a solution $(a,b)$ of the system \eqref{redu}, $\forall\sigma\in\mathbb{R}$
\begin{equation}\label{normhsa}
 \|a(t)\|_{h^\sigma_{p}}+\|b(t)\|_{h^\sigma_{p}}=\|a_0\|_{h^\sigma_{p}}+\|b_0\|_{h^\sigma_{p}}.
\end{equation}

Then, we remark that the system \eqref{redu} is Hamiltonian, for the Hamiltonian
\begin{equation*}\label{hamil}
H:=\displaystyle\sum_{(p,q,r,s)\in\Gamma_0}a_p\overline{a}_qb_r\overline{b}_s
=\displaystyle\sum\limits_{\{p,r\}=\{q,s\}}a_p\overline{a}_qb_r\overline{b}_s.
\end{equation*}
Indeed, with $H$ thus defined, we have the infinite system:
\begin{equation*}\begin{dcases}
i\partial_t{a_j}=\dfrac{\partial H}{\partial \overline{a}_j}, \quad 
-i\partial_t{\overline{a}_j}=\dfrac{\partial H}{\partial a_j}, \quad j\in\mathbb{Z},\\
i\partial_t b_j=\dfrac{\partial H}{\partial \overline{b}_j}, 
\quad -i\partial_t{\overline{b}_j}=\dfrac{\partial H}{\partial b_j}, \quad j\in\mathbb{Z}.
\end{dcases}\end{equation*}
The associated symplectic structure is given by $-i\displaystyle\sum_{j\in\mathbb{Z}}(\mathrm{d}a_j\wedge \mathrm{d}\overline{a}_j+
\mathrm{d}b_j\wedge \mathrm{d}\overline{b}_j)$. Thus, the Poisson bracket between two functions $f$ and $g$ of $(a,\overline{a},b,\overline{b})$ is given
by 
\begin{equation*}
 \{f,g\}=-i\displaystyle\sum\limits_{j\in\mathbb{Z}}\left(\dfrac{\partial f}{\partial a_j}\dfrac{\partial g}{\partial\overline{a_j}}-
\dfrac{\partial f}{\partial \overline{a_j}}\dfrac{\partial g}{\partial a_j}+\dfrac{\partial f}{\partial b_j}\dfrac{\partial g}{\partial\overline{b_j}}-
\dfrac{\partial f}{\partial \overline{b_j}}\dfrac{\partial g}{\partial b_j}\right).
\end{equation*}
In order to rewrite the Hamiltonian, we set 
\begin{equation*}
I:=\displaystyle\sum\limits_{n\in\mathbb{Z}}=|a_n|^2,\qquad J:=\displaystyle\sum\limits_{n\in\mathbb{Z}}=|b_n|^2,\qquad 
S:=\displaystyle\sum\limits_{n\in\mathbb{Z}}=a_n\overline{b_n}.  
\end{equation*}
Therefore, considering that $\{p,r\}=\{q,s\}$ in $\Gamma_0$, we can write the Hamiltonian $H$ as
\begin{equation}\label{HIJS}
H=IJ+|S|^2-\displaystyle\sum\limits_{n\in\mathbb{Z}}|a_n|^2|b_n|^2.
\end{equation}
With this new structure, we can give a new proof of estimates \eqref{conshs} and \eqref{normhsa}, for $n\in\mathbb{N}$ we have
\begin{align*}
\partial_t(|a_n|^2+|b_n|^2)&=\{|a_n|^2,H\}+\{|b_n|^2,H\}\\
&=-i\left(\overline{a_n}\dfrac{\partial H}{\partial \overline{a_n}}-a_n\dfrac{\partial H}{\partial a_n}\right)
-i\left(\overline{b_n}\dfrac{\partial H}{\partial \overline{b_n}}-b_n\dfrac{\partial H}{\partial b_n}\right)\\
&=-i\left(\overline{a_n}b_nS-a_n\overline{b_n}\overline{S}\right)-i\left(a_n\overline{b_n}\overline{S}-\overline{a_n}b_nS\right) \\
&=0.
\end{align*}
This implies the conservation of the $H^\sigma$ norms for the reduced resonant system \eqref{redu}, but also for the resonant system \eqref{res},
where the quantities $a_n, b_n, I, J$ and $S$ depend on $\xi$ too.
We are now able to prove the existence of the beating effect.

\subsection{Example of nonlinear dynamics: the beating effect}\label{ss45}
In this subsection, we adapt the proof of the beating effect of Gr\'ebert, Paturel and Thomann in \cite{GPT} to prove Theorem \ref{beating}.
We want to highlight an exchange between two modes of the solutions. According to this idea, we introduce the reduced space
\begin{equation*}
 \mathcal{J}_{p,q}=\{(a,b),a_n=\overline{a_n}=b_n=\overline{b_n}=0, n\notin \{p,q\}\},
\end{equation*}
and we note $\tilde{H}$ the reduced Hamiltonian defined by
\begin{equation*}
\tilde{H}=H_{\mathcal{J}_{p,q}}. 
\end{equation*}
For more simplicity, we work in the symplectic polar coordinates $(I_j,J_j,\theta_j,\varphi_j)$ defined by:
\begin{equation*}
 a_j=\sqrt{I_j}e^{i\theta_j}, \qquad b_j=\sqrt{J_j}e^{i\varphi_j}.
\end{equation*}
This is a symplectic change of variables because we have the relation
\begin{equation*}
\mathrm{d}\alpha\wedge \mathrm{d}\overline{\alpha}+\mathrm{d}\beta\wedge \mathrm{d}\overline{\beta}=i\left(\mathrm{d}\theta\wedge \mathrm{d}I
+\mathrm{d}\varphi\wedge\mathrm{d}J\right).
\end{equation*}
Therefore, the expression \eqref{HIJS} of the Hamiltonian $H$ becomes for the reduced case   
\begin{align*}
\tilde{H}&=(I_p+I_q)(J_p+J_q)+(a_p\overline{b}_p+a_q\overline{b}_q)(\overline{a}_pb_p+\overline{a}_qb_q)-(a_p^2b_p^2+a_q^2b_q^2)\\
&=(I_p+I_q)(J_p+J_q)+2(I_pI_qJ_pJ_q)^\frac12\cos(\Psi_0),
\end{align*}
where $\Psi_0:=\theta_q-\theta_p+\varphi_p-\varphi_q.$
We obtain thus the following Hamiltonian system
\begin{equation}\begin{dcases}\label{hamiltheta}                               
\partial_t\theta_j&=-\dfrac{\partial\tilde{H}}{\partial I_j}, \quad \partial_t I_j=\dfrac{\partial\tilde{H}}{\partial \theta_j}, \quad \quad j=p,q,\\
\partial_t\varphi_j&=-\dfrac{\partial\tilde{H}}{\partial J_j}, \quad \partial_t J_j=\dfrac{\partial\tilde{H}}{\partial \varphi_j}, \quad \quad j=p,q.
\end{dcases}
\end{equation}
Let us show that this system is completely integrable.
\begin{lem}
The system \eqref{hamiltheta} is completely integrable. Moreover, the following change of variables is symplectic:
\begin{equation*}\begin{dcases}
 K_0&=I_q,\quad K_1=I_q+I_p,\quad K_2=J_q+J_p,\quad K_3=I_q+J_q,\\
 \Psi_0&=\theta_q-\theta_p+\varphi_p-\varphi_q,\quad \Psi_1=\theta_p,\quad \Psi_2=\varphi_p,\quad \Psi_3=\varphi_q-\varphi_p.
\end{dcases}\end{equation*} 
\end{lem}

\begin{proof}
First, we easily see that $K_1, K_2$ and $K_3$ are constants of motion. Indeed, we can for example check for $K_1$ that 
\begin{align*}
\left\{K_1,\tilde{H}\right\}&=\sum\limits_{j=p,q}\left(\dfrac{\partial K_1}{\partial I_j}\dfrac{\partial \tilde{H}}{\partial \theta_j}-\dfrac{\partial K_1}
{\partial \theta_j}\dfrac{\partial \tilde{H}}{\partial I_j}+\dfrac{\partial K_1}{\partial J_j}\dfrac{\partial \tilde{H}}{\partial \varphi_j}-\dfrac{\partial K_1}
{\partial \varphi_j}\dfrac{\partial \tilde{H}}{\partial J_j}\right) \\
&=\dfrac{\partial \tilde{H}}{\partial \theta_p}+\dfrac{\partial \tilde{H}}{\partial \theta_q}=2\sqrt{I_pI_qJ_pJ_q}(\sin\psi_0-\sin\psi_0)=0.
\end{align*}
Then, $K_1,K_2$ and $K_3$ are independent of the angles $\theta_p, \theta_q, \varphi_p$ and $\varphi_q$. Therefore, they are in involution:
\begin{equation*}
\left\{K_1,K_2\right\}=\left\{K_1,K_3\right\}=\left\{K_2,K_3\right\}=0. 
\end{equation*}
Finally, $K_1,K_2$ and $K_3$ are  clearly independent (they do not depend on the same actions) and are independent with $\tilde{H}$ which is the only
to depend on an angle. Thus, the system is completely integrable.

Concerning the change of variables, we have
\begin{align*}
\mathrm{d}\Psi\wedge\mathrm{d}K&=\mathrm{d}\Psi_0\wedge\mathrm{d}K_0+\mathrm{d}\Psi_1\wedge\mathrm{d}K_1
+\mathrm{d}\Psi_2\wedge\mathrm{d}K_2+\mathrm{d}\Psi_3\wedge\mathrm{d}K_3\\
&=\mathrm{d}I_q\wedge \mathrm{d}(\theta_q-\theta_p+\varphi_p-\varphi_q+\theta_p+\varphi_q-\varphi_p)
+\mathrm{d}I_p\wedge \mathrm{d}\theta_p+\mathrm{d}J_p\wedge \mathrm{d}\varphi_p+\mathrm{d}J_q\wedge\mathrm{d}(\varphi_p+\varphi_q-\varphi_p)\\
&=dI\wedge d\theta+dJ\wedge d\varphi.
\end{align*}
Thus, the change of variables is symplectic.
\end{proof}

\begin{rem}
The fact that the system \eqref{hamiltheta} is integrable is not a surprise. Indeed, we can remark that the Hamiltonian $\tilde{H}$ depends only on one angle 
$(\Psi_0)$. This gives us three constants of motion, which with the Hamiltonian itself form a family of four independent constants of motions in involution. 
Therefore, the system \eqref{hamiltheta} is completely integrable. 
Another easy way to check this integrability is to consider the conservations of the mass and moment of the solutions,
given by the conservation of the $h^1$ norm of the solution thanks to equation \eqref{selfadR}. Therefore, the goal of this lemma is not to prove that the 
system is integrable, but to find a new explicit set of variable in order to obtain a simpler system.
\end{rem}

In this new system of coordinates, the Hamiltonian becomes 
\begin{equation*}
\tilde{H}=\tilde{H}(\Psi_0,K_0,K_1,K_2,K_3)=K_1K_2+2\left(K_0(K_1-K_0)(K_3-K_0)(K_2-K_3+K_0)\right)^\frac12\cos(\Psi_0).
\end{equation*}
$K_1,K_2$ and $K_3$ and constants of motion. Thus, for initial data of the size $\varepsilon$, we can fixed
\begin{equation*}
K_1=K_2=K_3=\varepsilon^2. 
\end{equation*}
Thus, we obtain a new Hamiltonian $\tilde{H}_0$ defined by
\begin{equation*}
\tilde{H}_0(\Psi_0,K_0)=\tilde{H}(\Psi_0,K_0,\varepsilon^2,\varepsilon^2,\varepsilon^2)=\varepsilon^4+2K_0(\varepsilon^2-K_0)\cos(\Psi_0).
\end{equation*}
To deal with all the $\varepsilon$ terms in the right-hand side of the previous equation, we make a change of unknown by setting
\begin{equation*}
\Psi_0(t)=:\Psi(\varepsilon^2t), \quad K_0(t)=:\varepsilon^2K(\varepsilon^2t).
\end{equation*}
We obtain the system:
\begin{equation}\label{Hstar}
\begin{dcases}                              
\dot{\Psi}&=2(2K-1)\cos\Psi=-\dfrac{\partial H_{\star}}{\partial K},\\
\dot{K}&=2K(K-1)\sin\Psi=\dfrac{\partial H_{\star}}{\partial \Psi},
\end{dcases}
\end{equation}
where the new Hamiltonian $H_{\star}$ is defined by
\begin{equation*}                              
H_{\star}=H_{\star}(\Psi,K):=2K(1-K)\cos \Psi.
\end{equation*}
Thanks to the conservation of the Hamiltonian, we can check that the velocity vector never cancels (for a good choice 
of initial data). Therefore we obtain periodic solutions, and the symmetries 
\begin{equation*}
 H_{\star}(-\Psi,K)=H_{\star}(\Psi,K), \quad H_{\star}(\Psi,1-K)=H_{\star}(\Psi,K),
\end{equation*}
give us informations on the solutions $(\Psi,K)$. Another way to check this periodicity of the solutions is to 
draw the phase portrait of the system defined by \eqref{Hstar}:
\begin{figure}[!ht]
\caption{Phase portrait of the Hamiltonian}
 \begin{center}
\includegraphics[scale=0.45]{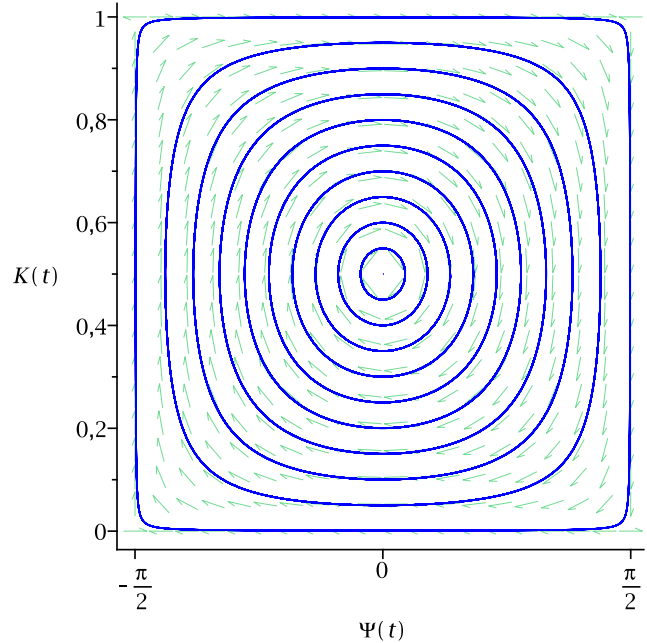}
\end{center}
\end{figure}
     
From this phase portrait, we deduce the following lemma:

\begin{lem}\label{lemphasport}
For all $\gamma\in(0,\frac12)$, there exists a time $T_\gamma$ such that
the system \eqref{Hstar} admits an orbit of period~$2T_\gamma$ with 
\begin{equation*}\begin{dcases}                               
(\Psi_\gamma(0),K_\gamma(0))=(0,\gamma),\\
(\Psi_\gamma(T_\gamma),K_\gamma(T_\gamma))=(0,1-\gamma).
\end{dcases}
\end{equation*} 
\end{lem}

\begin{rem}
We can show that the period $T_\gamma$ satisfies the bound
\begin{equation*}
 0<T_\gamma\lesssim |\ln(\gamma)|.
\end{equation*}
We refer to \cite[Theorem 1.1]{GPT} for the proof of this estimate.
\end{rem}

From Lemma \ref{lemphasport}, we have the beating effect in the following sense: 
for any couple of different integers $(p,q)$, there exists a solution
$(a,b)$ of the reduced resonant system \eqref{redu} such that
\begin{equation*}\begin{dcases}                               
|b_p(t)|^2=|a_q(t)|^2=\varepsilon^2K_\gamma(\varepsilon^2t),\\
|a_p(t)|^2=|b_q(t)|^2=\varepsilon^2(1-K_\gamma(\varepsilon^2t)).
\end{dcases}
\end{equation*}
Indeed, we can for example check the identity for $a_p$ by writing
\begin{equation*}
 |a_p(t)|^2=I_p(t)=K_1-K_0(t)=\varepsilon^2(1-K(\varepsilon^2t)).
\end{equation*}
We can remark that we have a $\varepsilon^2$ factor in the right-hand side which is not present in \cite{GPT}.
But in this article, this factor is present in the system.

\subsection{Beating effect and wave operator: proof of Theorem \ref{beating}}\label{ss46}

Let $I\subset\mathbb{R}$ be a bounded open interval, $(p,q)$ a couple of different integers and $0<\gamma<\frac12$.
We chose:
\begin{itemize}
 \item $\varepsilon>0$ small enough.
 \item $(a,b)$ a solution of the reduced resonant system \eqref{redu} with initial data
\begin{equation*}\begin{dcases}
  a_p(0)=b_q(0)=\varepsilon^2\gamma, \\   
  b_p(0)=a_q(0)=\varepsilon^2(1-\gamma),
\end{dcases}\end{equation*}
 that enjoys a beating effect between the modes $p$ and $q$, as constructed in the 
previous subsection.
 \item $\varphi\in\mathcal{S}(\mathbb{R})$ a compact supported function such that $\varphi\equiv1$ on $I$.
\end{itemize}
As in Subsection \ref{ss43}, we construct from $(a,b)$ a solution of the resonant system by choosing the initial data
\begin{equation}\begin{dcases}\label{defWU0beating}
  W_{U,0}=\check{\varphi}(x)\left(a_p(0)e^{ipy}+a_q(0)e^{iqy}\right)=\check{\varphi}(x)\varepsilon^2\left(\gamma e^{ipy}+(1-\gamma)e^{iqy}\right), \\   
  W_{V,0}=\check{\varphi}(x)\left(b_p(0)e^{ipy}+b_q(0)e^{iqy}\right)=\check{\varphi}(x)\varepsilon^2\left((1-\gamma) e^{ipy}+\gamma e^{iqy}\right).
\end{dcases}\end{equation}
By construction, the solution $(W_U,W_V)$ of the resonant system \eqref{res} with initial data $(W_{U,0},W_{V,0})$ satisfies
\begin{equation*}
 \mathcal{F}(W_{U})(t,\xi,k)=\varphi(\xi)a_k\left(\varphi(\xi)^2t\right),\qquad \qquad \mathcal{F}(W_{V})(t,\xi,k)=\varphi(\xi)b_k\left(\varphi(\xi)^2t\right).
\end{equation*}
By the choice of the solution $(a,b)$, and the fact that $\varphi\equiv1$ on $I$, this completes the proof of the first and second parts of Theorem \ref{beating}. 
For the third part of the theorem, in order to apply the modified wave operator from Section~\ref{modanswav}, we have to check that for $\varepsilon$ small enough, 
the constructed initial data $(W_{U,0},W_{V,0})$ satisfies
\begin{equation}\label{assumptionstocheck}
\|W_{U,0}\|_{S^+}+\|W_{V,0}\|_{S^+}\leq\varepsilon.
\end{equation}
For example, for $W_{U,0}$, we have by \eqref{defWU0beating}:
\begin{equation*}
 \left|\mathcal{F}(W_{U})(t,\xi,k)\right| \leq\varepsilon^2|\varphi(\xi)| \quad \text{if } k\in{p,q} \quad,  
 \left|\mathcal{F}(W_{U})(t,\xi,k)\right|=0 \quad \text{else}.
\end{equation*}
Therefore, using that:
\begin{itemize}
 \item there is only two modes for the periodic variable,
 \item $\varphi\in\mathcal{S}(\mathbb{R})$ is a compact supported function,
 \item there is a $\varepsilon^2$ factor and $\varepsilon$ is small enough,
\end{itemize}
the proof of \eqref{assumptionstocheck} and thus of Theorem \ref{beating} is completed.
Indeed, we recall that 
\begin{align*}
\|W_{U,0}\|_{S^+}&:=\|W_{U,0}\|_{H^N_{x,y}}+\|xW_{U,0}\|_{L^2_{x,y}}
+\|xW_{U,0}\|_{H^N_{x,y}}+\|x^2W_{U,0}\|_{L^2_{x,y}}\\
&\qquad+\|(1-\partial_{xx})^4W_{U,0}\|_{H^N_{x,y}}+\|x(1-\partial_{xx})^4W_{U,0}\|_{L^2_{x,y}}.
\end{align*}
For example, we can check that for $\varepsilon$ small enough,
\begin{align*}
 \|x(1-\partial_{xx})^4W_{U,0}\|^2_{L^2_{x,y}}&=\int_{\mathbb{R}}
 \varepsilon^4(\gamma^2+(1-\gamma)^2)\left|\partial_\xi\left((1+\xi)^4\varphi(\xi)\right)\right|^2\mathrm{d}\xi\leq \frac{\varepsilon^2}{144}.
\end{align*}

\qed

\subsection{Persistence of the beating effect and convolution potentials}\label{ss47}

Going back to the single cubic Schr\"odinger equation on $\mathbb{R}\times\mathbb{T}^d$:
\begin{equation*}
 i\partial_tU+\Delta_{\mathbb{R}\times\mathbb{T}^d}U=|U|^2U,
\end{equation*}
the method of Hani, Pausader, Tzvetkov and Visciglia in \cite{HPTV} allows to construct solutions that provide a growth of the Sobolev norms when $2\leq d\leq4$.
The idea is to find some particular solutions of the resonant equation and then to use the wave operator theorem.

In order to perturb the eigenvalues of the Laplacian, we can add a convolution potential:
\begin{equation*}
 i\partial_tU+\Delta_{\mathbb{R}\times\mathbb{T}^d}U+V\star U=|U|^2U.
\end{equation*}
In this case, Gr\'ebert, Paturel and Thomann show in \cite{GPT2} that for generic choice of convolution potential, the resonances are killed. 
For these potentials, there is no more any growth of the Sobolev norms. More precisely, the Sobolev norms of small solutions are asymptotically constant.

The natural question is thus to transpose the question about potentials to the system case. Are the potentials generically killing the beating effect ?
Are the Sobolev norms of the couples of solutions asymptotically constant ? The answer of the first question is given by a simple observation.
The key argument in \cite{GPT2} to kill the resonances and the growth of the Sobolev norms is to perturb the equation, thanks to the potentials, in order to 
obtain a smaller resonant set 
\begin{equation*}
\Gamma_{0,conv,d}=\left\{(p,q,r,s)\in(\mathbb{Z}^d)^4\;:\;p-q+r-s=0,\;\{|p|,|r|\}=\{|q|,|s|\}\right\}. 
\end{equation*}
In our case, we already have this relation. Indeed, we have 
\begin{equation*}
 \Gamma_{0}=\left\{(p,q,r,s)\in\mathbb{Z}^4\;:\;\{p,r\}=\{q,s\}\right\}=\Gamma_{0,conv,1}.
\end{equation*}
Therefore, we can conclude that the add of convolution potentials generically doesn't kill the beating effect. 
Thus, why are the Sobolev norms constant in \cite{GPT2} and not in our coupled system case ? 
This is just a consequence of the coupled effect: what we have in our case is the fact that the sums of the Sobolev norms of the couples of solutions
are asymptotically constant, which is the system equivalent of the result of~\cite{GPT2}.

\bigskip

\textit{Acknowledgments}: The author deeply thanks Beno\^it Gr\'ebert and Laurent~Thomann for several helps, advices and
their many suggestions. The author also would like to warmly thank Fr\'ed\'eric Bernicot and Olivier Pierre for several useful discussions. 

\bibliographystyle{abbrv}
\bibliography{bibli}

\end{document}